\documentclass[12pt]{amsart}

\usepackage{amssymb, amsmath, amsthm,stmaryrd}
\usepackage[backref]{hyperref}
\usepackage[alphabetic,backrefs,lite]{amsrefs}
\usepackage{amscd}   
\usepackage{fullpage}
\usepackage[all]{xy} 
\usepackage{MnSymbol,graphicx}
\usepackage{mathrsfs}

\usepackage{makecell}

\usepackage{tikz}
\usepackage{verbatim}

  \tikzstyle{block} = [rectangle, draw,
      text width=8em, text centered, minimum height=4em]
  \tikzstyle{line} = [draw, -latex']

\usepackage{amssymb, amsmath, amsthm,stmaryrd}
\usepackage[backref]{hyperref}
\usepackage[alphabetic,backrefs,lite]{amsrefs}
\usepackage{amscd}   
\usepackage{fullpage}
\usepackage[all]{xy} 
\usepackage{MnSymbol,graphicx}
\usepackage{mathrsfs,enumitem}
\usepackage{multirow,qtree}
\DeclareFontEncoding{OT2}{}{} 

\usepackage{tikz}

\newtheorem{lemma}{Lemma}[section]

\newtheorem{prop}[lemma]{Proposition}

\newtheorem{claim*}{Claim}
\newtheorem{thm}[lemma]{Theorem}

\theoremstyle{definition}

\newtheorem*{theoremnon}{Main Theorem}

\DeclareMathOperator{\invr}{inv}

\newcommand{\cS}{\mathcal{S}}
\newcommand{\cF}{\mathcal{F}}
\newcommand{\cA}{\mathcal{A}}
\newcommand{\cG}{\mathcal{G}}
\newcommand{\cP}{\mathcal{P}}
\newcommand{\cB}{\mathcal{B}}

\newcommand{\cD}{\mathcal{D}}
\newcommand{\cE}{\mathcal{E}}

\newcommand{\PP}{{\mathbb P}}
\newcommand{\C}{{\mathbb C}}
\newcommand{\F}{{\mathbb F}}

\newcommand{\Q}{{\mathbb Q}}
\newcommand{\R}{{\mathbb R}}
\newcommand{\Z}{{\mathbb Z}}
\newcommand{\cK}{\mathcal{K}}
\newcommand{\cU}{\mathcal{U}}
\newcommand{\cV}{\mathcal{V}}
\newcommand{\cL}{\mathcal{L}}

\newcommand{\T}{\theta}

\newcommand{\fp}{\mathfrak{p}}

\DeclareMathOperator{\Ker}{Ker}
\DeclareMathOperator{\Image}{Image}
\DeclareMathOperator{\genus}{genus}

\DeclareMathOperator{\Frob}{Frob}

\DeclareMathOperator{\inv}{inv}

\DeclareMathOperator{\Ext}{Ext}

\DeclareMathOperator{\Aut}{Aut}
\DeclareMathOperator{\Gal}{Gal}

\DeclareMathOperator{\divv}{div}

\DeclareMathOperator{\Div}{Div}
\DeclareMathOperator{\Pic}{Pic}

\DeclareMathOperator{\SL}{SL}
\DeclareMathOperator{\GL}{GL}
\DeclareMathOperator{\Sp}{Sp}

\usepackage{color}

\newcommand{\isom}{\cong}

\numberwithin{equation}{section}
\numberwithin{table}{section}

\title{Quadratic Points on  Modular Curves }

\author{Ekin Ozman}  
\author{Samir Siksek}

\address{Bogazici University\\
Department of Mathematics\\
Bebek, Istanbul, 34342 \\
Turkey}

\email{ekin.ozman@boun.edu.tr}

\address{Mathematics Institute\\
    University of Warwick\\
    CV4 7AL \\
    United Kingdom}

\email{samir.siksek@gmail.com}

\thanks{The first-named author is partially supported by Bogazici University Research Fund Grant Number 10842 and TUBITAK Research Grant 117F045. The second-named author is supported by an EPSRC LMF: L-Functions and Modular Forms Programme Grant EP/K034383/1.}
\date{\today}

\begin{document}

	\maketitle

	\begin{abstract}
In this paper we determine the quadratic points on the modular
curves $X_0(N)$, where the curve is non-hyperelliptic,
the genus is $3$, $4$ or $5$, and the Mordell--Weil group
of $J_0(N)$ is finite. The values of $N$ are	
$34$, $38$, $42$, $44$, $45$, $51$, $52$, $54$, $55$, $56$, 
$63$, $64$, $72$, $75$, $81$.

As well as determining the non-cuspidal quadratic points, we give 
the $j$-invariants of the elliptic curves parametrized by those
points, and determine if they have complex multiplication
or are quadratic $\Q$-curves.
	\end{abstract}

\section{Introduction}

Let $N$ be a positive integer. 
By the work of Mazur \cite{MazurEisen}, we have a complete understanding
of rational points on the modular curves $X_1(N)$; namely 
if $X_1(N)$ is of genus $\ge 1$, then the only rational points
are cuspidal. Merel's celebrated uniform boundedness theorem \cite{Merel}
asserts that for $d \ge 1$, there is some bound $B_d$ such that 
if $K$ is a number field of degree $\le d$, and $N \ge B_d$ is
prime, then the only $K$-rational points on $X_1(N)$ are cuspidal. 
There are more precise results for small fixed degrees $d$.
Kamienny \cite{Kamienny} showed that if $N \ge 17$ is prime
then $X_1(N)$ has no quadratic points, and the corresponding
result for cubic points was proved by Parent \cite{Parent1, Parent2}.
This has recently been extended to degrees $4$, $5$, $6$
by Derickx, Kamienny, Stein and Stoll \cite{DKSS}.

The situation concerning low degree points on the 
family $X_0(N)$ is much less happy. In fact we only
have complete results for the case of rational points.
Mazur \cite{Mazur}
proved that the only rational points on $X_0(N)$ 
are cusps when $N$ is prime and greater than $163$.
Later, these results were extended to composite levels 
and completed by 
Kenku (see \cite{Kenku} and the references therein). 
Results of Bars \cite{Bars} and Harris--Silverman \cite{HS}
assert that if $X_0(N)$ has genus $\ge 2$
then it has finitely many quadratic points,
except for 28 values of $N$.
A result of Aigner \cite{Aigner} gives all the solutions in quadratic fields to the Fermat equation $x^4 + y^4 = z^4$ which is isomorphic to $X_0(64)$. Recently, Bruin and Najman \cite{BN}
parametrizes all quadratic points on $X_0(N)$ explicitly
for those values of $N$ where $X_0(N)$ is hyperelliptic
and $J_0(N)$ has Mordell--Weil rank $0$; namely those
values of $N$ belonging to the set
\[
\{22, 23, 26, 28, 29, 30, 31, 33, 35, 39, 40, 41, 46, 47, 48, 50, 59, 71\}.
\]

In this paper we focus on non-hyperelliptic $X_0(N)$ 
of genera $3$, $4$, $5$ where the Mordell--Weil group
$J_0(N)(\Q)$ is finite. We determine the quadratic points
on these modular curves, and supply the modular interpretation
of the points. In the forthcoming second part of this work \cite{OS2}
we deal with those values of $N$ for which the genus is $3$, $4$, $5$
but $J_0(N)(\Q)$ is infinite, using a version of Chabauty
for symmetric powers of curves as in \cite{chabsym}. 
\begin{lemma}\label{lem:values}
The values of $N$ for which $X_0(N)$ is non-hyperelliptic,
of genus $g$ where $3 \le g \le 5$ and for which $J_0(N)(\Q)$ is finite
are 
\begin{description}[itemindent=1cm]
\item[genus $3$:] $34$, $45$, $64$;
\item[genus $4$:] $38$, $44$, $54$, $81$;
\item[genus $5$:] 
$42$, $51$, $52$, $55$, $56$, $63$,  $72$, $75$.
\end{description}
\end{lemma}

Let $X$ be a curve defined over $\Q$. A point $P \in X$ is called
quadratic if the field $\Q(P)$ is a quadratic extension of $\Q$.

\begin{theoremnon} 
For the values of $N$ listed in Lemma~\ref{lem:values}
the quadratic points on $X_0(N)$ are
as given in the tables of Section \ref{sec:tables}. 
\end{theoremnon}

For the non-cuspidal quadratic points, we compute 
$j$-invariants of the elliptic curves parametrized by them. 
In addition, we check whether those points are related via any 
Atkin-Lehner involutions and decide whether or not they have 
complex multiplication or are $\Q$-curves. 

One motivation for this work is the current interest
in the Fermat equation over quadratic fields
and similar Diophantine problems. The approach
via modularity and level-lowering requires
the irreducibility of the mod $p$ representation
of a Frey elliptic curve defined over
the given quadratic field, $K$ say. This Frey elliptic curve
often has extra level structure in the form
of a $K$-rational $2$ or $3$-isogeny. If the mod $p$
representation is reducible, then the
Frey curve gives rise to a $K$-rational point on $X_0(2p)$
or $X_0(3p)$, and thus having a parametrization
of quadratic points is useful in establishing irreducibility
for small values of $p$. The results of the current paper
have already proved useful in that context \cite{FermatSmall}.

\subsection*{A Theoretical Approach to the Problem}
Let $X/\Q$ be a non-hyperelliptic curve of genus $\ge 3$
with $J(\Q)$ finite where $J$ is the Jacobian of $X$,
and suppose for convenience that $X$ has at least one
rational point $P_0$.
For example $X$ could be any of the curves $X_0(N)$
for the values of $N$ in Lemma~\ref{lem:values}. 
There is a straightforward theoretical method 
(see for instance \cite{FLS}) of computing all
effective degree $2$ rational divisors on $X$,
and hence all points defined over quadratic extensions,
provided we are able to enumerate all the elements of $J(\Q)$. 
Let $X^{(2)}$ denote the second symmetric product of $X$.
 A $\Q$-rational point on $X^{(2)}$ can be represented by unordered pair 
$\{P_1,P_2\}$ where $P_1$, $P_2$ are either both
rational points on $X$, or a defined over a quadratic field
and Galois conjugate.
Let $D=P_1+P_2$ and let $\iota: X^{(2)}(\Q) \rightarrow J(\Q)$ 
be the map which sends $D$ to $[D-2P_0]$. Since $X$ is not hyperelliptic,
 $\iota$ is injective. 
By pulling back the finitely many points in $J(\Q)$, it is theoretically
possible to determine $X^{(2)}(\Q)$, hence the quadratic 
points of $X$ as follows. For any $\Q$-rational point on 
$X^{(2)}$, the corresponding degree $2$ divisor $D$ is linearly equivalent to 
$D^\prime+2P_0$ for some $[D^\prime]$ in $J(\Q)$ (where $D^\prime$
is rational degree $0$ divisor on $X$). 
Thus, for each $[D^\prime]$  in $J(\Q)$, we need to enumerate
the effective degree $2$ divisors linearly equivalent to $D^\prime+2P_0$.
For each $[D^\prime]$ in $J(\Q)$ we compute the Riemann-Roch space 
$L(D^\prime+ 2P_0)$. As the curve is non-hyperelliptic
the dimension of this space is either $0$ or $1$. If it has
dimension $0$ then there is no effective degree $2$ divisor
$D$ linearly equivalent to $D^\prime+2P_0$. If it has dimension $1$,
we let $f$ be a non-zero element of this space, and then
$D^\prime+2P_0+\divv(f)$ is the unique effective degree $2$ divisor
linearly equivalent to $D^\prime+2P_0$. There are potentially
two problems with this approach:
\begin{itemize}
\item It is often not convenient or practical to compute 
$J(\Q)=J(\Q)_{\mathrm{tors}}$.
\item Even if we can compute $J(\Q)$, it can be a large group,
and the Riemann-Roch computations might not be practical
for some of the more complicated elements $[D^\prime]$ of $J(\Q)$.
\end{itemize}

\subsection*{Our Approach}
For the modular curves of interest to us we 
first compute the rational cuspidal subgroup $C=C_0(N)(\Q)$ (see
below for definition) 
and bound its index inside $J(\Q)$, where $J=J_0(N)$. 
We therefore know a positive integer
$I$ such that $I \cdot J(\Q) \in C$. It follows that the 
effective degree $2$ divisors $D$ we seek satisfy $[D-2P_0]=I \cdot [D^\prime]$
where $[D^\prime] \in J(\Q)$. We then employ a version of the Mordell--Weil
sieve 
to help us eliminate most possibilities for $D^\prime$. 
Only then
do we use Riemann--Roch to recover the divisors $D$.

\subsection*{The Generalized Ogg Conjecture}
For now let $N$ be any positive
integer. Let $C_0(N)$ be the subgroup of $J_0(N)(\overline{\Q})$
generated by classes of differences of cusps;
this is known as the cuspidal subgroup. Write
$C_0(N)(\Q)$ for the subgroup of $C_0(N)$ of points
stable under
the action of $\Gal(\overline{\Q}/\Q)$; this
is known as the rational cuspidal subgroup,
and is contained in $J_0(N)(\Q)$. The Manin--Drinfeld
theorem \cite{Manin}, \cite{Drinfeld}
in fact asserts that $C_0(N) \subseteq J_0(N)(\overline{\Q})_{\mathrm{tors}}$,
and thus $C_0(N)(\Q) \subseteq J_0(N)(\Q)_{\mathrm{tors}}$.
A conjecture of Ogg, proved by Mazur \cite{MazurEisen}, says that 
the $C_0(N)(\Q)=J_0(N)(\Q)_{\mathrm{tors}}$ for $N$ prime.
A \lq\lq generalized Ogg conjecture\rq\rq\ (stated for example
in \cite{Ren}) asserts that this equality holds for all positive $N$.
Our computations verify the conjecture for most of the values of 
$N$ mentioned in the statement of Main Theorem.
\begin{thm}\label{thm:genOgg}
The generalized Ogg conjecture holds for $N=34, 38, 44, 45, 51, 52, 54, 56, 64, 81$.
\end{thm}
For recent partial results towards the generalized Ogg conjecture
see \cite{Ren}, \cite{Ling}, \cite{Ohta}, \cite{Yoo}.

\bigskip

The second-named author would like to thank
Martin Derickx, Steve Donnelly and  Derek Holt for useful
conversations. The authors would like to thank Ozlem Ejder and
Jeremy Rouse for pointing out that $X_0(64)$
is isomorphic to the Fermat quartic and pointing us 
in the direction of Aigner's paper \cite{Aigner}. The authors are grateful 
to the referee for suggesting several significant
improvements and simplifications. The authors would also like
to thank the \emph{Istanbul Center for Mathematical Sciences (IMBM)}
for hosting their collaboration during April 2018.


\section{Choices of $X_0(N)$: Proof of Lemma~\ref{lem:values}}\label{sec:choices}
Ogg \cite{Ogg} has shown that the values of $N$
for which $X_0(N)$ is hyperelliptic are
\[
22,\, 23,\, 26,\, 28,\, 29,\, 30,\, 31,\, 33,\, 35,\, 37,\, 39,\, 
40,\ 41,\, 46,
47,\, 48,\ 50,\, 59,\, 71.
\]
The genus of $X_0(N)$ grows with $N$ and there are only
finitely many values of $N$ for any given genus $g$.
Using the explicit formula for the genus (e.g. ~
\cite[Section 3.9]{DS}), which is
also the dimension formula for $S_2(N)$,
we found the values of $N$ for which $X_0(N)$
is non-hyperelliptic and has genus $3 \le g \le 5$
to be
\begin{description}[itemindent=1cm]
\item[genus $3$:] $34$, $43$, $45$, $64$;
\item[genus $4$:] $38$, $44$, $53$, $54$, $61$, $81$;
\item[genus $5$:] $42$, 
$51$, $52$, $55$, $56$, $57$, $63$, $65$, $67$, $72$, $73$, $75$.
\end{description}
We need to decide on the values of $N$ in this
list for which $J_0(N)$ has rank $0$. For this
we briefly recall standard facts about the decomposition
of $J_0(N)$ as a product of abelian varieties of $\GL_2$-type;
for more details see for example Stein's thesis \cite{SteinsThesis}.

Let $f_1,\dotsc,f_k$ be representatives of the Galois-conjugacy
classes of Hecke eigenforms in $S_2(N)$. Let $K_i$ be
the (totally real) number field generated by the coefficients
of $f_i$, and let $d_i$ be its degree. Attached to each 
$f_i$ (or more precisely to its Galois-conjugacy class) is
an abelian variety $\cA_i/\Q$ of dimension $d_i$
whose endomorphism ring contains an order in $K_i$.
In particular the rank of $\cA_i(\Q)$ is a multiple of $d_i$.
The modular Jacobian $J_0(N)$ is 
isogenous to $\cA_1 \times \cdots \times \cA_k$.
Thus $J_0(N)(\Q)$ is finite if and only if the $\cA_i$
all have rank $0$. Let $\cA$ be any of the $\cA_i$. We write
$L(\cA,s)$ for the $L$-function of $\cA$. 
The conjecture of Birch and Swinnerton-Dyer asserts that
the rank of $\cA(\Q)$ is equal to the order of vanishing
of $L(\cA,s)$ at $s=1$. A deep theorem of 
Kolyvagin and Logachev \cite{KL} asserts that
if $L(\cA,1) \ne 0$ then $\cA(\Q)$ has rank $0$. 
In fact an algorithm of Stein \cite[Chapter 3]{SteinsThesis} 
allow us to compute the exact value $L(\cA,1)/\Omega_\cA \in \Q$
where $\Omega_\cA$ is the real volume of $\cA$.
Using the \texttt{Magma} \lq\lq modular abelian
varieties package\rq\rq, which is an implementation by Stein of 
his algorithms \cite{Stein},
\cite{SteinsThesis} we computed the ratios
$L(\cA,1)/\Omega_\cA$ and found them all to be non-zero
for the values of $N$ listed in the statement of Lemma \ref{lem:values}.
Thus we know that $J(\Q)$ is finite for those values.
It remains to show that $J(\Q)$ is infinite for
the remaining values
$43$, $53$, $61$, $57$, $65$,  $67$, $73$; for these
we claim the ranks respectively are $1$, $1$, $1$, $1$, $1$, $2$, $2$.
For each of these values of $N$, there is only one factor $\cA$
for which $L(\cA,1)=0$, and thus the rank of $J_0(N)$ is
equal to the rank of $\cA$. For $N=43$, $53$, $61$, $57$, $65$
this factor $\cA$ happens to be a rank $1$ elliptic curve,
completing the proof in those cases.
For $N=67$ and $73$ this factor is simply $J_0^+(N)$ which in both
cases is $2$-dimensional, and it remains to
show that this has rank $2$ in both cases. For both values, a model for 
the genus $2$ curve $X_0^+(N)$
is given by Galbraith \cite[page 43]{Galbraith}.
Using the \texttt{Magma} implementation
of Stoll's $2$-descent algorithm \cite{Stoll}
we checked that the rank of $J_0^+(N)$ is indeed $2$ in both cases.

\section{Computing equations for the $X_0(N)$, Atkin--Lehner Involutions,
$j$-maps, and cusps}     \label{section:models} 

Let $N$ be such that the modular curve $X_0(N)$ has genus $g \ge 3$
and is non-hyperelliptic. Then the canonical map embeds $X_0(N)$
into projective space $\PP^{g-1}$ as a smooth curve of degree $2g-2$, and it is 
this model that we work with. An algorithm for writing
down the model is given by Galbraith \cite{Galbraith}. Although
equations for $X_0(N)$ for most of the $N$ we consider are 
already given by Galbraith, we needed to redo his computations
so that we can explicitly construct the Atkin--Lehner involutions
on the models, and also the $j$-map 
$X_0(N) \rightarrow X(1)$. We start by briefly recalling
Galbraith's method.

Let $S_2(N)$ be the space
of weight $2$ cuspforms of level $N$ with $q$-expansion
coefficients belonging to $\Q$; this
has dimension $g=\genus(X_0(N))$.  The cuspforms in $S_2(N)$
can be identified with the regular differentials
on $X_0(N)/\Q$ via the map $f(q) \mapsto 2 \pi i f(q) dq/q$.
Fix a $\Q$-basis for $f_0,\dotsc,f_{g-1}$ for  $S_2(N)$;
such a basis maybe computed as $q$-expansions to 
a desired precision
 via the modular symbols algorithm \cite{Cre}, \cite{Stein}.
Now let $F \in \Q[x_0,\dotsc,x_{g-1}]$ be homogenous of degree $d$.
Then $F(f_0,\dotsc,f_{g-1})$ is a cuspform of weight $2d$
and level $N$. 
Write 
\[
I=[\SL_2(\Z) : \Gamma_0(N)]=N \prod_{p \mid N} \left(1+\frac{1}{p} \right).
\]
By Sturm's Theorem (e.g. ~ \cite[Theorem 9.18]{Stein}),
$F(f_0,\dotsc,f_{g-1})=0$ if and only if 
the $q$-expansion $F(f_0(q),\dotsc,f_{g-1}(q))=O(q^r)$
with $r=\lfloor d I/6 \rfloor+1$. Thus determining the vector
space of all homogenous $F$ of degree $d$ such that
$F(f_0,\dotsc,f_{g-1})=0$ is a straightforward linear algebra
computation, and carrying this out  for $d \mid (2g-2)$ ensures
that we have a system of equations that cuts out a
model for $X_0(N)$ in $\PP^{g-1}$. For the values of $N$
in Lemma~\ref{lem:values} we carried this out
and, conveniently, found a model for $X_0(N)$
that has good reduction away from the primes dividing $N$.
The equations for these models are given in our tables at the end.

\medskip

Next we would like to work out the Atkin--Lehner involutions
on $X_0(N)$. Let $m \mid N$ such that $\gcd(m,N/m)=1$
and $m \ne 1$. The modular symbols algorithm gives 
the action of Atkin--Lehner operator $w_m$ as a linear
operator of order $2$ on $S_2(N)$ and hence as an
order $2$ matrix of size $g \times g$
with entries in $\Q$. Now the linear automorphism
on $\PP^{g-1}$ induced by this matrix restricts to the
Atkin--Lehner involution $w_m$ on $X_0(N)$.

\medskip

Next we describe how we obtain the 
map $j : X_0(N) \rightarrow X(1)$; this is not
described in \cite{Galbraith} but similar computations
are found in \cite{BC} and \cite{FLS}. We start with
largest divisor $n \mid N$ for which we already know the following:
\begin{enumerate}
\item[(i)] equations for $X_0(n)$;
\item[(ii)] generators $u_1,\dotsc,u_s$ for the function field
of $X_0(n)$ together with their $q$-expansions at the cusp
at infinity (these will in general be Laurent series);
\item[(iii)] the map $X_0(n) \rightarrow X(1)$.
\end{enumerate}
In fact, for all values of $N$ that we consider, we found that 
the \texttt{Magma} \lq\lq small modular curves\rq\rq\ package
gives (i), (ii), (iii) with $n$ the largest proper divisor
of $N$.  The idea is to construct the degeneracy map 
$X_0(N) \rightarrow X_0(n)$
whence composition with $X_0(n) \rightarrow X(1)$ gives the desired
$j: X_0(N) \rightarrow X(1)$. For this it is sufficient
to construct the pull-backs of $u_1,\dotsc,u_s$
to $X_0(N)$ which we denote by $U_1,\dotsc,U_s$. 
Fix $1 \le i \le s$ and let $U=U_i$ and $u=u_i$.
Then $U$ is a rational function on $X_0(N)$, and hence 
can be written as 
\[
U=\frac{F(x_0,\dotsc,x_{g-1})}{G(x_0,\dotsc,x_{g-1})}
\]
where $F$, $G$ are homogeneous in $\Q[x_0,\dotsc,x_{g-1}]$
of equal degree $d$. These satisfy
\begin{equation}\label{eqn:linsys}
F(f_0(q),\dotsc,f_{g-1}(q))-u(q) \cdot G(f_0(q),\dotsc,f_{g-1}(q))=0
\end{equation}
where $u(q)$ is the known $q$-expansion for $u$.
Fix a degree $d$. Let $V_d$ be the vector space of 
all homogeneous $\Q[x_0,\dotsc,x_{g-1}]$ of degree $d$ and let
$V_d^\prime$ be the subspace belonging to the homogenous ideal 
generated by the equations of $X_0(N)$. Note that
if $H \in V_d^\prime$ then $H(f_0(q),\dotsc,f_{g-1}(q))=0$.
Thus we may think of  \eqref{eqn:linsys} as a linear equation
in $(F,G) \in V_d/V_d^\prime \times V_d/V_d^\prime$, and we would like
to find a non-trivial solution (for a suitable choice of $d$).
Fixing $d$ and a large \lq precision\rq\ $m$ we consider the linear system
of equations
\begin{equation}\label{eqn:linsysm}
F(f_0(q),\dotsc,f_{g-1}(q))-u(q) G(f_0(q),\dotsc,f_{g-1}(q)) = O(q^m),
\end{equation}
in $(F,G) \in V_d/V_d^\prime \times V_d/V_d^\prime$. By choosing
$d$ large enough we were always able to find a non-trivial 
solution $(F,G)$ and thus a highly plausible guess for $U=F/G$
(we took $m=500$ in all our examples). Next we checked
that the guesses $U_1,\dotsc,U_s$ do in fact give a map 
$X_0(N) \rightarrow X_0(n)$ and we composed this
with the known $X_0(n) \rightarrow X(1)$ 
to obtain a proposed $j$-function $X_0(N) \rightarrow X(1)$
 which for now we denote
by $j^\prime$. We did not prove the correctness
of the degeneracy map $X_0(N) \rightarrow X_0(n)$ (which
we do not use later), but we did
prove the correctness of the proposed $j$-function on $X_0(N)$
as we now explain. For now we think of $j$ and $j^\prime$
as elements of the function field of $X_0(N)$. 
The above procedure gives 
$j^\prime=H_1(x_0,\dotsc,x_{g-1})/H_2(x_0,\dotsc,x_{g-1})$ where $H_1$, $H_2$ are homogeneous in $\Q[x_0,\dotsc,x_{g-1}]$ of equal degree. 
The $q$-expansion for $j^\prime$ is given by 
$j^\prime(q)=H_1(f_1(q),\dotsc,f_{g-1}(q))/H_2(f_1(q),\dotsc,f_{g-1}(q))$
and by computing enough terms 
we checked that $j(q)-j^\prime(q)=O(q^m)$ some large $m$, where
\[
j(q)=\frac{1}{q}+744+196884 q + 21493760 q^2 + 864299970 q^3+\cdots
\]
is the usual expansion of the $j$-function. Note that here
we cannot simply apply Sturm's Theorem to deduce $j=j^\prime$
as we have not shown that $j-j^\prime$ is a modular form
(i.e. holomorphic at all the cusps of $X_0(N)$), so we
adopt a different approach. Let $D$ and $D^\prime$ be
the divisor of poles for $j$ and $j^\prime$ respectively.
By \cite[pages 106--107]{DS}, for $N>2$, 
\[
\deg(D)=\deg(j)=\frac{N^2}{\varphi(N)}
\cdot \prod_{p \mid N} \left( 1-\frac{1}{p^2}\right),
\]
where $\varphi$ denotes the Euler totient-function.
Since we know $j^\prime$ we can compute $D^\prime$
explicitly and we checked that $\deg(D^\prime)=\deg(D)$
in all cases. Now if $j \ne j^\prime$ then 
the divisor of poles for $j-j^\prime$ is bounded
by $D+D^\prime$ and so has degree at most $2\deg(D)$.
But the order of vanishing of $j-j^\prime$
at the cusp $\infty$ is at least $m$. Since the divisor
of zeros has the same degree as the divisor of poles
we deduce that $m \le 2\deg(D)$.
In all cases $m$ exceeded   $2\deg(D)+1$ by a huge
margin, proving $j=j^\prime$.

We note in passing that the \texttt{Magma}
\lq\lq small modular curves\rq\rq\ package
is a wonderful resource, but that the justification
for the modular curves data is only very briefly sketched
in the \texttt{Magma} handbook. Whilst we make use
of this package to guess the $j$-function on $X_0(N)$
our subsequent proof of the correctness of our guess is
independent of it.

\medskip

Finally, as we have have the $j$-map we can compute
the cusps; these are merely the poles of $j$.



\section{An Extension Problem for Finite Abelian Groups}\label{sec:Ext}
Let $C$, $A_1$, $A_2$ be finite abelian groups and suppose
$\iota_j : C \rightarrow A_j$ are injective homomorphisms.
In this section we address the following question: \emph{is there is an isomorphism $\psi : A_1 \rightarrow A_2$ such
that $\psi \circ \iota_1=\iota_2$?} This a problem we will need to 
address later on where $C$ happens to be the rational
cuspidal subgroup of $J_0(N)$ and $A_1$, $A_2$ are candidates
for $J_0(N)(\Q)_{\mathrm{tors}}$ obtained from local
information. Let $p$ be a prime, and write $C[p^\infty]$, $A_1[p^\infty]$ and $A_2[p^\infty]$
for the $p$-power torsion in $C$, $A_1$, $A_2$. Clearly the question has a positive
answer if and only if the corresponding question for $C[p^\infty]$, $A_1[p^\infty]$ and $A_2[p^\infty]$
has a positive answer for every prime $p$ dividing the orders of the groups $A_1$, $A_2$.
Thus we may suppose that $C$, $A_1$, $A_2$ are $p$-power torsion finite abelian groups
for some prime $p$.

Of course the question has a negative answer
if $A_1$ is not isomorphic to $A_2$. Thus first we 
make sure that $A_1$, $A_2$ are isomorphic and write
down an explicit isomorphism $\psi_0 : A_1 \rightarrow A_2$.
Next we can write down the automorphism group of $A_2$
and deduce the set of all isomorphisms $A_1 \rightarrow A_2$;
any such isomorphism will be a composition of $\psi_0$
with an automorphism of $A_2$. These steps can be carried out 
using for example algorithms explained in \cite{CH} and implemented
in \texttt{Magma}. Now we may simply test the isomorphisms $\psi : A_1 \rightarrow A_2$
and see if there is one that satisfies $\psi \circ \iota_1=\iota_2$.
This strategy does provide a theoretical answer to our
question. In our application we have found it impractical as the automorphism groups $\Aut(A_i)$
are enormous. As an illustration we point out that if $A=(\Z/p\Z)^n$,
then $\Aut(A) \cong \GL_n(\F_p)$. Thus whilst $\#A=p^n$, we have $\#\Aut(A)=(p^n-1)(p^n-p) \cdots (p^n-p^{n-1})$.

Write $B_i=A_i/\iota_i(C)$ and let $\pi_i : A_i \rightarrow B_i$ be the quotient maps.
Any isomorphism  $\psi : A_1 \rightarrow A_2$ 
satisfying $\psi \circ \iota_1=\iota_2$
induces an isomorphism $\mu : B_1 \rightarrow B_2$
that makes the diagram \eqref{eqn:CD} commute, where the two rows are exact.
\begin{equation}\label{eqn:CD}
\begin{CD}
0 @>>> C @> \iota_1 >> A_1 @> \pi_1 >> B_1 @>>>0\\
@.    @|              @V \psi VV       @V \mu VV \\
0@>>> C @> \iota_2 >> A_2 @> \pi_2 >> B_2 @>>>0\\
\end{CD}
\end{equation}
Thus we know that our question has a negative answer if $B_1$, $B_2$ are not isomorphic.
We suppose that they are isomorphic and we enumerate all isomorphisms $\mu : B_1 \rightarrow B_2$
(by computing the automorphism group of $B_2$). In our application the groups $B_i$ tend
to be rather small and so there are far fewer isomorphisms $B_1 \rightarrow B_2$
than isomorphisms $A_1 \rightarrow A_2$. For each $\mu$ we now ask the following: is there an isomorphism
$\psi: A_1 \rightarrow A_2$ that makes the diagram \eqref{eqn:CD} commute. 
In essence we can interpret both exact sequences as extensions of $B_2$ by $C$
and we are asking if they are equivalent extensions. However we are interested
in answering this question in the category of finite abelian groups and would like
to avoid computing $\Ext(B_2,C)$ (which classifies all extensions of $B_2$ by $C$ including the
non-abelian ones) as well as avoiding the computation of the images of the two sequences in this group.
The following proposition
gives us an efficient way of answering the question in the category of finite abelian groups.
\begin{prop}\label{prop:lifting}
Let 
\[
0 \rightarrow C \xrightarrow{\iota_i} A_i \xrightarrow{\pi_i} B_i \rightarrow 0
\]
be exact sequences of finite abelian groups, for $i=1$, $2$. Let $\mu : B_1 \rightarrow B_2$
be an isomorphism. 
Let $x_1,\dotsc,x_r$ be any elements of $A_1$ such that
$A_1$ is generated by $\iota_1(C)$ together with $x_1,\dotsc,x_r$. 
Let $y_1,\dotsc,y_r \in A_2$ satisfy $\pi_2(y_j)=\mu(\pi_1(x_j))$
for $j=1,\dotsc,r$ (these must exist as $\pi_2$ is surjective). 
Write $\mathbf{x}=(x_1,\dotsc,x_r) \in A_1^r$ and $\mathbf{y}=(y_1,\dotsc,y_r) \in A_2^r$.
Let $\cA=\Z^r \times C$. Let $\tau: \cA \rightarrow A_1$ be given by 
\[
\tau\left(\mathbf{m},c \right)=\mathbf{m} \cdot \mathbf{x}+\iota_1(c) \qquad \text{for $\mathbf{m} \in \Z^r$ and $c \in C$};
\]
here $(m_1,\dotsc,m_r) \cdot (x_1,\dotsc,x_r)$ is shorthand for the linear combination $m_1 x_1+\cdots+m_r x_r$.
Let $(\mathbf{n}_1,c_1),\dotsc,(\mathbf{n}_s,c_s)$ be a set of generators for the kernel of $\tau$. 
Define a homomorphism
\[
\eta : C^r \rightarrow A_2^s, \qquad \mathbf{t} \mapsto 
\left(\iota_2(\mathbf{n_1} \cdot \mathbf{t}), \dotsc, \iota_2(\mathbf{n_s} \cdot \mathbf{t}) \right),
\]
and let
\[
\kappa=\left(\mathbf{n}_1 \cdot \mathbf{y}+\iota_2(c_1), \dotsc,\mathbf{n}_s\cdot \mathbf{y}+\iota_2(c_s) \right) \in A_2^s.
\]
There exists an isomorphism $\psi : A_1 \rightarrow A_2$ making the diagram \eqref{eqn:CD}
commute if and only if $\kappa \in \Image(\eta)$.
\end{prop}
\begin{proof}
Suppose $\kappa$ belongs to the image of $\eta$. 
We will show the existence of a homomorphism $\psi : A_1 \rightarrow A_2$ making the 
diagram \eqref{eqn:CD} commute. It then easily follows that $\psi$ must be an isomorphism (this
fact is known as the \emph{short-five lemma}).
As $\kappa$ is in the image of $\eta$, so is $-\kappa$.
Thus there 
is some $\mathbf{t}=(t_1,\dotsc,t_r) \in C^r$ such that 
\begin{equation}\label{eqn:kercond}
\mathbf{n}_j \cdot \mathbf{y}+\iota_2(c_j)=-\iota_2(\mathbf{n}_j \cdot \mathbf{t}), \qquad j=1,\dotsc,s.
\end{equation}
Let 
\[
\sigma : \cA \rightarrow A_2, \qquad 
\sigma\left(\mathbf{m},c \right)=\mathbf{m} \cdot \mathbf{y}+\iota_2(\mathbf{m}\cdot \mathbf{t})+\iota_2(c),
\qquad \text{for $\mathbf{m} \in \Z^r$ and $c \in C$}.
\]
Recall that $(\mathbf{n}_1,c_1),\dotsc,(\mathbf{n}_s,c_s)$ are generators for the kernel of $\tau: \cA \rightarrow A_1$.
Condition \eqref{eqn:kercond} ensures that the kernel of $\tau$ is contained in the
kernel of $\sigma: \cA \rightarrow A_2$. Thus we obtain a well-defined homomorphism
$\cA/\Ker(\tau) \rightarrow \cA/\Ker(\sigma) \rightarrow A_2$. By hypothesis
$A_1$ is generated by $x_1,\dotsc,x_r$ and $\iota_1(C)$ thus $\tau : \cA \rightarrow A_1$
is surjective. We let $\psi$ be the composition
\[
A_1 \xrightarrow{\sim} \cA/\Ker(\tau) \rightarrow \cA/\Ker(\sigma) \rightarrow A_2.
\]
It follows from the definitions of $\tau$ and $\sigma$ that $\psi$ sends
$\iota_1(c)$ to $\iota_2(c)$ for any $c \in C$, and sends
$x_j$ to $y_j+\iota_2(t_j)$ for $j=1,\dots,r$. In particular, the left-hand square of \eqref{eqn:CD}
commutes. Moreover 
\[
\pi_2(\psi(x_j))=\pi_2(y_j+\iota_2(t_j))=\pi_2(y_j)=\mu(\pi_1(x_j))
\]
where the last equality comes from our original definition of the $y_j$.
Since $x_1,\dotsc,x_j$ together with $\iota_1(C)$ generate $A_1$ the right-hand square of \eqref{eqn:CD}
also commutes.

Now conversely suppose there is an isomorphism $\psi: A_1 \rightarrow A_2$ making \eqref{eqn:CD} commute.
Thus $\pi_2(\psi(x_j))=\mu(\pi_1(x_j))=\pi_2(y_j)$. By the exactness of the bottom row,
 $\psi(x_j)=y_j+\iota_2(t_j)$ for some $t_j \in C$. Now let $(\mathbf{n}_i,c_i)$ be as in the statement of the theorem.
Thus $\mathbf{n}_i \cdot \mathbf{x}+\iota_1(c_i)=0$. Letting $\mathbf{t}=(t_1,\dotsc,t_r)$
and applying $\psi$ we have 
\[
\mathbf{n}_i \cdot \mathbf{y}+\iota_2(\mathbf{n}_i \cdot \mathbf{t})+\iota_2(c_i)=0.
\]
It follows that $\kappa=\eta(-\mathbf{t})$ completing the proof.
\end{proof}


\section{The Mordell--Weil Information}
Now let $N$ be one of the values of $N$ in Lemma~\ref{lem:values}.
In this section we explain how compute the structure
of the rational cuspidal subgroup $C_0(N)(\Q)$
as well as deducing
a small integer $I$ such that $I \cdot J_0(N)(\Q)_{\mathrm{tors}} 
\subseteq C_0(N)(\Q)$.

\subsection{\textbf{Magma} computations in Jacobians of curves}
It is more convenient
computationally to work with places of a curve than with 
points defined over various extensions of the base field.
Let $X$ be a smooth projective curve over a perfect field $F$.
A place $\cP$ of $X$ is simply a set 
of distinct points $\{P_1,\dotsc,P_n\} \subset X(\overline{F})$ 
that is stable under the action of $\Gal(\overline{F}/F)$
and forms a single orbit under that action. The size $n$
is called the degree of $\cP$. It also happens to be the
degree of $F(P_i)/F$ for any $i$. It is often convenient
to think of $\cP$ as the effective degree $n$ rational divisor 
$P_1+\cdots+P_n \in \Div(X/F)$,
and indeed any rational divisor can be written uniquely
as an integral linear combination of places.
We shall make use of \lq\lq algebraic function fields\rq\rq\
package
within \texttt{Magma}, which carries out computations in
$\Div(X/F)$ and more importantly in the quotient 
$\Pic(X/F)$ and its degree $0$ part $\Pic^0(X/F) \cong J(F)$
(for this latter isomorphism to hold we need to suppose
the existence of a degree $1$ place on $X$). For the theory
behind these algorithms see Hess' paper \cite{Hess}. When $F=\Q$ or 
a number field some of the computations we would like to do
become impractical, particularly in the genus $5$ cases.
However \texttt{Magma} computations in 
$\Pic^0(X/F)$ are much more efficient when $F$ is a finite field,
and in fact the \texttt{Magma} implementation of Hess' algorithm
does compute the structure of $\Pic^0(X/F) \cong J(F)$ in all the 
cases of interest to us, where $F$ is a finite
field of characteristic $p$, with $p$ a prime
of good reduction for our model $X=X_0(N)$.
Our strategy is to carry out as much of the computations
over finite fields as possible. A particularly useful
fact for us is the following. Let $X$ be curve
defined over a number field $K$, let $p>2$
be a rational prime, let $\fp$ be a prime of $K$ above
$p$ of ramification degree $1$ and of good reduction for $X$.
Then a theorem of Katz \cite[appendix]{Katz} asserts that
the composition of natural maps 
$J(K)_\mathrm{tors} \hookrightarrow J(K) \rightarrow J(\F_\fp)$
is an injection. In particular we may identify $J(K)_\mathrm{tors}$
as a subgroup of $J(\F_\fp)$.

\subsection{A closer look at the rational cuspidal subgroup}
\label{sub:rcs}
To ease notation we shall write $X$ and $J$ for $X_0(N)$
and $J_0(N)$.
Let $d^2$ be the largest square divisor of $N$, and let
$K=\Q(\zeta_d)$ be the $d$-th cyclotomic field. The cusps
of $X$ are all defined over $K$. Let these
be $P_0,\dotsc,P_k$ with $P_0$ defined over
$\Q$ (there are always at least two
rational cusps: the $\infty$ cusp and the $0$ cusp).  
Henceforth $P_0$ will be the base-point for the
Abel--Jacobi map $X \rightarrow J$. Let
\[
\cD=\sum_{i=1}^{k} \Z \cdot (P_i-P_0). 
\]
This is the subgroup of $\Div^0(X/K)$ supported on the cusps.
Write $G=\Gal(K/\Q)$; this acts naturally on $\cD$
and we denote by $\cD^G$ the subgroup of divisors 
that are stable under $G$ (these are the degree $0$,
$\Q$-rational divisors supported on the cusps). 
Write 
\[
C^\prime=\{ [D] \; : \; D \in \cD^G\}
\]
for the image of $\cD$ in $J(\Q)$. Now $G$
also acts on the subgroup
\[
\cE=\sum_{i=1}^k \Z \cdot [P_i-P_0]
\]
of $J(K)$. The group $\cE^G \subset J(\Q)$ is precisely the rational
cuspidal subgroup $C=C_0(N)(\Q)$. Of course $C^\prime$ is contained
inside $C$, and it is natural to ask if they are equal. 
\begin{lemma}\label{lem:cusps}
Let $N$ be one of the values in Lemma~\ref{lem:values}.
Then $C=C^\prime$. In other words, every degree $0$
rational divisor class supported on the cusps is
the class of a degree $0$ rational divisor supported
on the cusps.
\end{lemma} 
The lemma will bring some simplifications to our later
computations comparing $C$ with the torsion subgroup.
We note in passing that as $X(\Q) \ne \emptyset$,
it is known \cite[Section 3]{PoonenSchaefer} that every 
degree $0$ rational divisor class is the class of a
degree $0$ rational divisor. However, applying
this to a degree $0$ rational divisor class supported
on the cusps does yield a degree $0$ rational divisor
defining the same class, but that divisor need not
be supported on the cusps.
\begin{proof}[Proof of Lemma~\ref{lem:cusps}]
Since we have the cusps as points on $X$ with coordinates
in $K$, we can compute $\cE^G$.

Now let $p\nmid 2N$ be a rational prime  and let $\fp$
be a prime of $K$ above $p$. In particular $p \nmid d$
and so $\fp$ is an unramified prime. It follows from the
aforementioned theorem of Katz that the reduction
modulo $\fp$ map $\pi \; : \; J(K)_{\mathrm{tors}} \rightarrow J(\F_\fp)$ is injective.
For $\sigma \in G$ we let 
\[
\mu_\sigma \; : \; \cD \rightarrow J(\F_\fp), \qquad
D \mapsto \pi(D^\sigma-D).
\]
It follows from the injectivity of $\pi$ that
$[D]^\sigma=[D]$ if and only if  $\mu_\sigma(D)=0$.
Let
\[
\cF=\bigcap_{\sigma \in G} \Ker(\mu_\sigma).
\]
This is precisely the subgroup of $\cD$ of divisors representing
rational divisor classes. The image of $\cF$ in $J(\F_\fp)$
lands in fact inside $J(\F_p)$ and is isomorphic to $C$.
The image of $\cE^G$ inside $J(\F_\fp)$ is contained
in the image of $\cF$ and is isomorphic to $C^\prime$. For each $N$
we made a suitable choice of $p$, $\fp$ and computed
both these images and checked that they are equal.
Thus $C=C^\prime$. 
\end{proof}
Having established the equality $C=C^\prime$, we have 
another more convenient way of thinking of $C$.
Let $\cP_0,\dotsc,\cP_r$ be the cusp places on $X/\Q$,
with $\cP_0=P_0$ as before a cusp place of degree $1$.
From the equality $C=C^\prime$ we now know that
\[
C=\sum_{i=1}^r \Z (\cP_i - \deg(\cP_i) \cdot \cP_0).
\]
Thus for any $p \nmid N$, to compute the image of 
$C$ in $J(\F_p)$ we merely take the subgroup 
generated by the reductions of $\cP_i-\deg(\cP_i) \cdot \cP_0$..

\subsection{The real torsion subgroup of $J_0(N)$}
Let $N$ be a positive integer and let $g$ be the genus of $X_0(N)$.
The torsion subgroup of $J_0(N)(\C)$
is isomorphic to $(\Q/\Z)^{2g}$. So the group $J_0(N)(\Q)_\mathrm{tors}$
isomorphic to a product of $\Z/d_1\Z \times \cdots \times \Z/d_{2g}\Z$ with
$d_1 \mid d_2 \mid \cdots \mid d_{2g}$. However 
$J_0(N)(\Q)_\mathrm{tors}$ is contained in the torsion subgroup 
of $J_0(N)(\R)$ and we can use this to deduce that 
$d_1,\dotsc,d_g \in \{1,2\}$, and often in fact to cut
down the number of possibilities for $d_1,\dotsc,d_g$
as we shall see below. 

We shall need the following theorem of Snowden \cite{Snowden},
which tells us the number connected components of $X_0(N)(\R)$.
\begin{thm}[Snowden]\label{thm:Snowden}
Let $N$ be a positive integer. If $N$ is a power of $2$
then $X_0(N)$ has one real component. Otherwise
let $n$ be the number of odd prime divisors of $N$.
Let $\epsilon=1$ if $8 \mid N$ and $\epsilon=0$ otherwise.
Then $X_0(N)$ has $2^{n+\epsilon-1}$ real components.
\end{thm}

We shall also need the following well-known theorem,
which perhaps first appears in a paper of 
Gross and Harris \cite{GH}.
\begin{thm}\label{thm:GH}
Let $X/\R$ be a smooth curve with $X(\R) \ne \emptyset$.
Let $g$ be the genus of $X$, $m$ the number of its real
components and $J$ its Jacobian. Then
\[
J(\R) \cong (\R/\Z)^g \times (\Z/2\Z)^{m-1}.
\]
Thus
\[
J(\R)_{\mathrm{tors}} \cong (\Q/\Z)^g \times (\Z/2\Z)^{m-1}.
\]
\end{thm}
\begin{proof}
By \cite[Proposition 3.2]{GH}, the number
of connected components of $J(\R)$ is $2^{m-1}$.
The theorem follows from \cite[Section 1]{GH}.
\end{proof}

\subsection{Computing the possibilities for $J_0(N)(\Q)$}  
We return to $N$ being one of the values in Lemma~\ref{lem:values},
and we continue writing $X=X_0(N)$, $J=J_0(N)$ and $C=C_0(N)(\Q)$.
Recall that
$J(\Q)$ is finite for all values of $N$ we are considering.
Thus $C \subseteq J(\Q)_{\mathrm{tors}}=J(\Q)$.
In particular, the reduction modulo $p$ map $J(\Q) \rightarrow J(\F_p)$
is injective for $p\nmid 2N$. 
Let $\cA_p^{\prime}$ be set of homomorphisms $\iota: C \rightarrow A$
where $A$ is a subgroup of $J(\F_p)$ containing the image 
of $C$ under the reduction mod $p$ map, and $\iota$ is the restriction
of the mod $p$ map to $C$. 
Thus we know that for \emph{some} 
$\iota \in \cA_p^\prime$,
we have a commutative diagram 
\begin{equation}\label{eqn:CJA}
\xymatrix{
C \ar@{->}[dd]_{\iota}\ar@{^(->}[rr] & & J(\Q)\ar@{->}[dd]^{\text{red}}\ar@{=}[ddll]^{\mu}  \\
& & \\
A\ar@{^(->}[rr]  & &  J(\F_p) }
\end{equation}
where $\mu$ is an isomorphism. 
Let $g$ be the genus of $X$, and $m$ be the number of real components
of $J$ which maybe computed from Theorem~\ref{thm:Snowden}.
By Theorem~\ref{thm:GH}, we know that 
\[
J(\Q) \cong \Z/d_1\Z \times \cdots \times \Z/d_k \Z, \qquad
d_1 \mid d_2 \mid \cdots \mid d_k
\]
where $k \le g$ or
$g+1 \le k \le g+m-1$ and $d_1,\dotsc,d_{k-g} \in \{1,2\}$.
Thus we eliminate from $\cA_p^\prime$ all $\iota : C \rightarrow A$
where the isomorphism class of $A$ is incompatible with this information,
to obtain a subset $\cA_p$. 

Let $p_1,\dotsc,p_s$ be distinct primes
$\nmid 2 p N$. We let
$\cA_{p;\, p_1,\dotsc,p_{s}}$ be the
set of $\iota : C \rightarrow A$ in $\cA_p$ such that the following holds:
for each $p^\prime \in \{p_1,\dotsc,p_s\}$ there is some
$\iota^\prime : C \rightarrow A^\prime$ in $\cA_{p^\prime}$
and an isomorphism $\psi: A \rightarrow A^\prime$ making the diagram

\[
\xymatrix{
C \ar@{->}[drr]_{\iota^{\prime}}\ar@{->}[rr]^{\iota} & & A\ar@{->}[d]^{\psi} \\
& & A^{\prime} }
\]
commute. The existence of the isomorphism can be efficiently
decided using the
method explained in Section~\ref{sec:Ext}.
It is clear that there must be some $\iota: C \rightarrow A$
in $\cA_{p;p_1\dotsc,p_s}$ and an isomorphism
$\mu : A \rightarrow J(\Q)$
such that the diagram \eqref{eqn:CJA} commutes. Observe
that $\cA_{p;p_1,\dotsc,p_s}$ must contain some
$\iota_0: C \rightarrow A_0$ where $A_0$ is the image of $C$
under the reduction mod $p$ map.

Our hope is to find suitable $p$, $p_1,\dotsc,p_s$
such that $\cA_{p;p_1,\dotsc,p_s}$ contains precisely
one element, in which case this is necessarily $\iota_0$,
and we can then deduce that $J(\Q)=C$. In any case
we know that $J(\Q)/C$ is isomorphic to the cokernel
of some $\iota$ in $\cA_{p;p_1,\dotsc,p_s}$ which
allows us to deduce a positive integer $I$ such that
$I \cdot J(\Q) \subseteq C$.

\begin{lemma}\label{lem:C}
Let $N$ be one of the values given in Lemma~\ref{lem:values}.
Then $C$ and $J(\Q)/C$ are as given in the tables
of Section~\ref{sec:tables}.
\end{lemma}
\begin{proof}[Proof of Lemma~\ref{lem:C} and Theorem~\ref{thm:genOgg}]
We wrote a \texttt{Magma} script which for each value of $N$
computed $\cA_{p;p_1,\dotsc,p_s}$ where $p$ is the smallest
rational prime not dividing $2N$, and $p_1,\dotsc,p_s$
are the primes $\le 17$ not dividing $2pN$. This allowed us
to deduce the information in the tables except
for two cases, where $N=45$, $64$. For those two 
cases were able to improve on the information
given in by this method and deduce that $J(\Q)=C$.
We explain this below in Sections~\ref{sub:45} and~\ref{sub:64}.
\end{proof}

\subsection{The Mordell--Weil Group for $J_0(45)$}\label{sub:45}
Let $N=45$. The procedure explained above tells us that
$C \cong \Z/2\Z \times \Z/4\Z \times \Z/8\Z$, and
\[
J(\Q) \cong \Z/2\Z \times \Z/4\Z \times \Z/8\Z \quad \text{or} \quad 
\Z/2\Z \times \Z/2\Z \times \Z/4\Z \times \Z/8\Z
\quad \text{or} \quad 
\Z/2\Z \times \Z/4\Z \times \Z/4\Z \times \Z/8\Z .
\]
We would like to show that $J(\Q)=C$, and for this it 
is enough to show that $J(\Q)[2]=C[2]=(\Z/2\Z)^3$.
However for all primes $p \nmid N$ we tried
$J(\F_p)[2]=(\Z/2\Z)^6$ or $(\Z/2\Z)^4$,
and so it does not seem to be possible to prove
the desired conclusion using reduction modulo primes.
Instead we will compute the entire mod $2$ representation
of $J$
\[
\overline{\rho}_{J,2} \; : \; \Gal(\overline{\Q}/\Q) \rightarrow \Sp_6(\F_2)
\]
 and use this to deduce that $J(\Q)[2]=(\Z/2\Z)^3$.

The curve $X=X_0(45)$ is a smooth plane 
quartic. Our model for it is
\[
X \; : \; x^3 z - x^2 y^2 + x y z^2 - y^3 z - 5 z^4=0
\]
in $\PP^2$. A procedure for computing the mod $2$ representation
of Jacobians of smooth plane quartics is explained
by Bruin, Poonen and Stoll \cite[Section 12]{BPS} and we apply that
method to our situation. We wrote down the equations for the 
$28$ bitangents to $X$. We found that the field of definition of these
bitangents is $K=\Q(\sqrt{-3},\sqrt{5})$ and so $\Q(J[2])=K$.
In particular $\overline{\rho}_{J,2}$ factors through $\Gal(K/\Q)$,
and we continue to denote the corresponding representation
$\Gal(K/\Q) \rightarrow \Sp_6(\F_2)$ by $\overline{\rho}_{J,2}$.
Each bitangent $L$ meets $X$ in a divisor $(L.X)$ which has the form $2 D_L$
where $D_L$ is an effective degree $2$ divisor. 
We wrote down the set $\Sigma$ of all quadruples $\{L_1,\dotsc,L_4\}$
such that $D_{L_1}+\cdots+D_{L_4} \sim 2 \omega_X$, where $\omega_X$
is the canonical divisor on $X$. As predicted by \cite{BPS} this set
$\Sigma$ has cardinality $315$. Next we construct the graph $\cG$
whose vertices are the quaruples $Q \in \Sigma$, and where $Q \ne Q^\prime$
are connected by an edge if and only if $Q \cap Q^\prime \ne \emptyset$.
We computed the automorphism group $\Aut(\cG)$ using \texttt{Magma} (this routine
is an implementation of the algorithm described in \cite{McKay}), and found it
to be isomorphic to $\Sp_6(\F_2)$ as predicted by \cite{BPS}. Now
the action of $\Gal(K/\Q)$ on the lines naturally gives 
a representation $\Gal(K/\Q) \rightarrow \Aut(\cG) \cong \Sp_6(\F_2)$,
and this is precisely $\overline{\rho}=\overline{\rho}_{J,2} : \Gal(K/\Q) \rightarrow \Sp_6(\F_2)$ (up to conjugation inside
$\Sp_6(\F_2)$). 
Let $\tau_1$, $\tau_2$, $\tau_3$ be the elements of $\Gal(K/\Q)$ satisfying 
\[
\begin{cases}
\tau_1(\sqrt{-3})=\sqrt{-3}\\
\tau_1(\sqrt{5})=-\sqrt{5}\\
\end{cases}, 
\qquad
\begin{cases}
\tau_2(\sqrt{-3})=-\sqrt{-3}\\
\tau_2(\sqrt{5})=\sqrt{5}\\
\end{cases}, 
\qquad
\tau_3=\tau_1 \tau_2.
\]
Write $M_i=\overline{\rho}(\tau_i)$. We found that
\[
M_1=
\begin{pmatrix}
1 & 0 & 0 & 0 & 0 & 0\\
0 & 0 & 1 & 1 & 0 & 1\\
0 & 1 & 0 & 1 & 0 & 1\\
0 & 0 & 0 & 1 & 1 & 1\\
0 & 0 & 0 & 0 & 0 & 1\\
0 & 0 & 0 & 0 & 1 & 0\\
\end{pmatrix}, \quad
M_2=
\begin{pmatrix}
1 & 0 & 0 & 0 & 0 & 0\\
1 & 0 & 0 & 0 & 0 & 1\\
1 & 0 & 0 & 1 & 1 & 1\\
1 & 1 & 0 & 1 & 0 & 1\\
1 & 0 & 1 & 1 & 0 & 1\\
1 & 1 & 0 & 0 & 0 & 0\\
\end{pmatrix}, \quad
M_3=
\begin{pmatrix}
1 & 0 & 0 & 0 & 0 & 0\\
1 & 0 & 0 & 0 & 1 & 0\\
1 & 0 & 0 & 1 & 0 & 0\\
1 & 0 & 1 & 0 & 0 & 0\\
1 & 1 & 0 & 0 & 0 & 0\\
1 & 0 & 1 & 1 & 0 & 1\\
\end{pmatrix}.
\]
Write $V_i=\Ker(M_i-I)$ where $I \in \Sp_6(\F_2)$ is the identity matrix. Thus $V_i$
is isomorphic to the subspace of $J[2]$ fixed by $\tau_i$. We found that
$V_1$, $V_2$, $V_3$ are $4$-dimensional, but $V_1 \cap V_2 \cap V_3$
is $3$-dimensional. This proves that $J(\Q)[2]\cong (\Z/2\Z)^3$, and completes
the proof that $J(\Q)=C$. 

It is worth observing that if $p$
is any prime of good reduction, then the Frobenius element $\Frob_p \in \Gal(K/\Q)$
must either be the identity or one of the $\tau_i$; in the former case $J(\F_p)[2] 
\cong (\Z/2\Z)^6$ and in the latter case $J(\F_p)[2] \cong (\Z/2\Z)^4$.
This explains why we have been unable to use the reduction mod $p$
maps to precisely pin down $J(\Q)[2]$.

\subsection{The Mordell--Weil group for $J_0(64)$}\label{sub:64}
Let $N=64$. The procedure explained above gives us 
\[
C\cong \Z/2\Z \times (\Z/4\Z)^2,
\]
and 
\[
J(\Q) \cong \Z/2\Z \times (\Z/4\Z)^2,
\qquad \text{or}
\qquad
(\Z/4\Z)^3.
\]
We want to show that $J(\Q)=C$. For this it is enough
to show that $J(\Q)[4]=C$.
The cusps of $X_0(64)$ are defined over $K=\Q(\zeta_8)=\Q(i,\sqrt{2})$.
Let $L=\Q(\sqrt{2})$. Let $G=\Gal(K/\Q)$ and $H=\Gal(K/L) \subset G$.
The method explained in Subsection~\ref{sub:rcs}
computes $C_0(64)(K)$ (the full cuspidal group)
and then $C=C_0(64)(\Q)$ is obtained by taking $G$-invariants.
Instead we take $H$-invariants to obtain
$C_0(64)(L)$ and find $C_0(64)(L) \cong (\Z/4\Z)^3$. 
Note that this is contained in $J(L)[4]$. However,
from Theorems~\ref{thm:Snowden} and~\ref{thm:GH} 
we know that $J(\R)[4] \cong (\Z/4\Z)^3$. 
As $L \subset \R$ we deduce $J(L)[4]=C_0(64)(L) \cong (\Z/4\Z)^4$.
Hence $J(\Q)[4] \subseteq C_0(64)(K)$. 
Now taking $G$-invariants we have $J(\Q)[4] \subseteq C_0(64)(K)^G=C$.
This proves that $J(\Q)=C$.
    
\section{The Mordell--Weil Sieve}
Let $X/\Q$ be a curve of genus $\ge 3$ with Jacobian $J$,
and suppose $X$ is not hyperelliptic. In this section
we explain a version of the Mordell--Weil sieve for quadratic points
on $X$, under the assumption that $J(\Q)$ has rank $0$,
but without assuming full knowledge of $J(\Q)$.
Let $P_0 \in X(\Q)$. We use this to fix a map 
$X^{(2)} \rightarrow J$ given by $D \mapsto [D-2P_0]$.

Let $\cK$ be a (known) set of rational effective divisors of degree $2$.
Let $G$ be a subgroup of $J(\Q)$ and $I$ a positive integer
such that $I \cdot J(\Q) \subseteq G$; again we assume that $G$ and $I$
are known, but that $J(\Q)$ is perhaps unknown. We will use our partial
knowledge of the Mordell--Weil group 
to sieve for unknown rational effetive degree $2$ divisors.
Let $p \ge 3$ be a prime of good reduction for $X$. Let
\[
\cV_p=\{\tilde{D} \in X^{(2)}(\F_p) \; : \; D \in \cK\},
\qquad
\cU_p=X^{(2)}(\F_p) \setminus \cV_p.
\]
\begin{lemma}\label{lem:nonhyp}
Let $D^\prime \in X^{(2)}(\Q)$ and suppose
$\tilde{D^\prime} \in \cV_p$. Then
$D^\prime \in \cK$.
\end{lemma}
\begin{proof}
By defintion of $ \cV_p$ we have $\tilde{D^\prime}=\tilde{D}$
for some $D \in \cK$. Thus the divisor class $[D^\prime-D]$
is in the kernel of the reduction map $J(\Q) \rightarrow J(\F_p)$.
However, $J(\Q)$ is torsion. By the injectivity of torsion
\cite[appendix]{Katz} under mod $p$ reduction
we conclude that $D\sim D^\prime$. It will be enough to show
that $D=D^\prime$. Suppose otherwise, then the Riemann--Roch
space $\cL(D)$ has dimension at least $2$. Let $f \in \cL(D)$
be a non-constant function. Then $f : X \rightarrow \PP^1$
has degree $2$ contradicting the assumption that $X$
is non-hyperelliptic.
\end{proof}

\begin{lemma}\label{lem:mwsieve}
Let $p_1,\dots,p_r$ be primes $\ge 3$ of good reduction for $X$. Let
$\phi_i : G \rightarrow J(\F_{p_i})$ be the composition for the inclusion $G \hookrightarrow J(\Q)$
with the reduction modulo $p_i$ map $J(\Q) \rightarrow J(\F_{p_i})$. If $D \in X^{(2)}(\Q)\setminus \cK$
then
\begin{equation}\label{eqn:mwsieve}
I\cdot [D-2 P_0] \in \bigcap_{i=1}^r \phi_i^{-1} \left( \left\{I \cdot [ \tilde{D}-2 \tilde{P_0}] \; : \; 
\tilde{D} \in \cU_{p_i} \right\}\right) .
\end{equation}
\end{lemma}
\begin{proof}
Since $I \cdot J(\Q) \subseteq G$, we know that $I \cdot [D- 2P_0] \in G$.
By Lemma~\ref{lem:nonhyp}, we know that the reduction $\tilde{D}$ of $D$ modulo $p_i$ belongs to
$\cU_{p_i}$. The proof follows easily.
\end{proof}

\section{Proof of Main Theorem}
Let $N$ be one of the values in Lemma~\ref{lem:values}.
Thanks to Lemma~\ref{lem:C}, we know a set that contains
all the possibilities for $J(\Q)/C$. We take $I$
to be the least common multiple of the exponents
of these groups. Thus $I \cdot J(\Q) \subseteq C$.
In each case we know the rational points
on $X$ and thus we have some collection
$\cK_0$ of effective degree $2$ divisors
of the form $P+Q$ where $P$, $Q$ are rational
points on $X$. We searched for quadratic points on
$P$ on $X$ and took $\cK$ to the the union of $\cK_0$
together with $P+P^\sigma$ where $P$ runs through the quadratic
points we have found, and $P^\sigma$ is Galois conjugate of $P$.
We took $\cK$ as our known set of degree $2$ divisors on $X$.
We then applied Lemma~\ref{lem:mwsieve} for a suitable
choice of primes $p_1,\dotsc,p_r \ge 3$ of good reduction,
and with $G=C$,
to deduce a subset of $\cS \subseteq J(\Q)$ given by \eqref{eqn:mwsieve}
that contains the possibilities for $I \cdot [D-2P_0]$
for $D \in J(\Q)\setminus \cK$. We found, for $N \ne 42$, $72$,
that $\cS =\emptyset$, and thus $J(\Q)=\cK$. The
two values $N=42$, $72$ needed virtually identical
arguments to complete the proof that $J(\Q)=\cK$.
We illustrate this now by giving a detailed account
of the computation for $X_0(72)$. 

\bigskip

The curve $X_0(72)$ has $8$ cusps of degree $1$ and $4$ cusp pairs defined over quadratic fields.
Let these be $P_0,\dotsc,P_7$ and $Q_i$, $Q_i^\prime$ where $i=1,2,3,4$ and 
$Q_i$, $Q_i^\prime$ are Galois conjugates.
The modular symbols algorithm shows that $J_0(72)$ is isogenous
to a product $E_1^2 \times E_2^2 \times E_3$ where $E_1$, $E_2$, $E_3$
are respectively the elliptic curves with Cremona labels \texttt{24A1}, \texttt{36A1},
\texttt{72A1}. As these three elliptic curves have rank $0$ so does $J_0(72)$.
We write $X=X_0(72)$ and $J=J_0(72)$.
Let $\cK$ 
consist of the known degree $2$ effective rational divisors:
$Q_i+Q_i^\prime$ for $i=1,\dotsc,4$ and $P_i+P_j$ where $0 \le i,j \le 7$.
Thus $\cK$ has $40$ elements. From our tables in Section~\ref{sec:tables}
we have
\[
C=\Z/2\Z \times \Z/4\Z \times \Z/12\Z \times \Z/12\Z, \qquad
J(\Q)/C=0 \quad \text{or} \quad \Z/2\Z.
\]
We take $G=C$ and $I=2$.
We applied Lemma~\ref{lem:mwsieve} with just one prime $p=5$.
We found that $\#X^{(2)}(\F_5)=64$. Of these $64$ divisors, $40$ are reductions of elements in $\cK$
and the $2 \times 5$ matrix $A(\tilde{D})$ has rank $2$ for all $40$. Thus $\# \cU_5$
has $24$ elements. However, only two elements of the set $\{2[\tilde{D}-2\tilde{P_0}] \; : \; \tilde{D} \in \cU_5\}$
are in the image of $\phi_5 \; : \; G \hookrightarrow J(\F_5)$. We let $A_1$, $A_2$ be their preimages in $G$ (which
we represent as divisors of degree $0$ on $X$). Thus if $D \in X^{(2)}(\Q) \setminus \cK$ then
$2D\sim A_i+4P_0$. We found that the Riemann--Roch spaces $L(A_i+4P_0)$ are both $2$-dimensional.
Suppose $i=1$ and let $f$, $g$ be a $\Q$-basis for $L(A_1+4P_0)$. Thus $2D=A_1+\divv(\alpha f+ \beta g)$
for some $(\alpha,\beta) \in \PP^1(\Q)$. We consider the  $1$-dimensional family of $0$-dimensional
subschemes $A_1+\divv(\alpha f+\beta g)$ of $X^{(2)}$ parametrized by $(\alpha,\beta) \in \PP^1$.
Computing and factoring the discriminant of this subscheme (as a homogeneous expression in $\alpha$, $\beta$)
shows that none of the elements of this $1$-dimensional family has the form $2D$ with $D \in X^{(2)}(\Q)$.
This shows that no element $D \in X^{(2)}(\Q)$ satisfies $2D \sim A_1+4P_0$. An identical argument
also allows us to deduce a contradiction for $i=2$. 
This shows that $X^{(2)}(\Q)=\cK$. Therefore there are no
non-cuspidal quadratic points on $X_0(72)$. 

\bigskip

\noindent \textbf{Remark.}
In the above computation for $X_0(72)$
we have applied the Mordell--Weil sieve with just one prime $p=5$.
In fact we have found in this case that using additional primes does not allow us to eliminate $A_1$, $A_2$
using just the Mordell--Weil sieve and it is instructive to ponder the reason for this.
In both cases $i=1$, $2$, the discriminant mentioned above has three quadratic factors 
which give divisors $D$ belonging to $X^{(2)}(\Q(\sqrt{-1}))$,
$X^{(2)}(\Q(\sqrt{3}))$ and $X^{(2)}(\Q(\sqrt{-3}))$ such that $2D \sim A_i+4P_0$. 
Since any prime $p \ge 5$ must split in at least one
of the fields $\Q(\sqrt{-1})$, $\Q(\sqrt{3})$, $\Q(\sqrt{-3})$ it follows that for any such $p$ (and for $i=1$, $2$)
there
is some $\tilde{D} \in X^{(2)}(\F_p)$ such that $2 \tilde{D} \sim \tilde{A_i}+4 \tilde{P}_0$. This fact can
be easily used to show that $A_1$, $A_2$ belong to the intersection in \eqref{eqn:mwsieve}
regardless of which primes $p_1,\dots,p_r \ge 5$ are chosen.

\section{Tables}\label{sec:tables}


\begin{table}[!htbp]
	\caption{\boldmath $X_0(34)$}
	    \begin{equation*}
	                \begin{split}
	                \text{Genus: \;} & 3 \\
\text{	Model: \;} & x^3z - x^2y^2 - 3x^2z^2 + 2xz^3 + 
	    3xy^2z - 3xyz^2 + 4xz^3 - y^4 + 
	    4y^3z - 6x^2z^2 + 4yz^3 - 2z^4 \\
J_0(34)(\Q)& = C \isom \Z/4\Z \times \Z/12\Z 
\end{split}
\end{equation*}

	\begin{tabular}{cccccc}
		\hline
		Name & $\T^2$ & Coordinates &$ j$-invariant & CM by & $\Q$-curve\\
		\hline \hline &&&&& \\
		$P_1$ & -1& $(\T+1,0,1)$ &  287496 & -16 & YES\\
		$P_2$ & -1& $(\frac{\T+1}{2},\frac{\T+1}{2},1)$ & 1728 & -4 & YES\\[1ex]  
		$P_3$ & -1 & $(\T,-\T,1)$ & 1728 & -4 & YES\\ [1ex]  
		$P_4$ & -2 & $(\frac{\T}{2},-\frac{\T}{2},1)$ & 8000  & -8 & YES\\[1ex]
		$P_5$ & -15 &  $(\frac{\T+11}{8},\frac{1}{2} ,1)$  & $\frac{2041\T + 11779 }{8}$ & NO & YES \\[1ex]
	
		$P_6$ & -15 & $(\frac{\T+23}{16},\frac{\T+7}{16},1)$ &  $\frac{-53184785340479\T - 7319387769191}{34359738368}  $& NO & YES \\[2ex]
	\hline
	\end{tabular} 
	
	\vspace{0.5cm}
\begin{tabular}{ccccc}
		\xymatrix{{P_1} \ar@{-}[d]_{w_{17}}\ar@{-}[r]^{w_{34}}\ar@{-}[dr] & {P_2} \ar@{-}[d]\ar@{-}[dl] _{ w_{2}\;\; \; }\\
			{P_1^{\sigma}} \ar@{-}[r] & {P_2^\sigma} } &
		\xymatrix{{P_5} \ar@{-}[d]_{w_{2}}\ar@{-}[r]^{w_{34}}\ar@{-}[dr] & {P_6} \ar@{-}[d]\ar@{-}[dl]_{w_{17}\;\; \; }\\
			{P_5^\sigma} \ar@{-}[r] & {P_6^\sigma}      } &
		\xymatrix{{P_3} \ar@{=}[d]^{w_{17},w_{34}} \ar@{-}[r]^{w_2} & {P_3}\ar@{=}[d]^{w_{17},w_{34}}
			\\ {P_3^\sigma}   \ar@{-}[r]^{w_2}  & {P_3^\sigma}  }  &
		\xymatrix{{P_4} \ar@{=}[d]^{w_{17},w_{34}} \ar@{-}[r]^{w_2} & {P_4}\ar@{=}[d]^{w_{17},w_{34}}
			\\ {P_4^\sigma}   \ar@{-}[r]^{w_2}  & {P_4^\sigma}  }  \\                 
		\end{tabular}  
         \end{table}
         
  \vspace{2cm}
         

\begin{table}[!htbp]
                \caption{\boldmath $X_0(38)$}
                \begin{equation*}
                \begin{split}
                \text{Genus:\;}& 4 \\
                \text{Model:\;} & x_1x_3 - x_2^2 - x_2x_4 - x_3^2 - x_3x_4 - x_4^2, \\
                & x_1^2x_4 + x_1x_4^2 - x_2^3 + 3x_2^2x_3 + 2x_2^2x_4 - 
                                    3x_2x_3^2 - 4x_2x_3x_4 - 2x_2x_4^2 + x_3^3 + 2x_3^2x_4 + 
                                    2x_3x_4^2 + x_4^3 \\
            J_0(38)(\Q) & = C \isom \Z/3\Z \times \Z/45\Z 
            \end{split} \end{equation*}

                \begin{tabular}{cccccc}
               \hline
                Name & $\T^2 $& Coordinates & $j$-Invariant & CM by & $\Q$-curve\\
                \hline \hline &&&&& \\
                $P_1$ & -3& $(\frac{\T+1}{2} , 0 ,\frac{\T-1}{2} , 1)$ & 0 & -3 & YES\\[1ex]
                $P_2$ & -3 &  $(0 , \frac{-\T+1}{2}, 0 , 1)$ & 54000 & -12 & YES\\[1ex]
                $P_3$ & -2& $(\frac{\T+1}{3}, \frac{-\T+1}{3}, \frac{\T-1}{3} , 1)$ & 8000&  -8 & YES\\[1ex]
                \hline
                \end{tabular} 
                
 \begin{tabular}{cc}
                          
        \xymatrix{
                        {P_1} \ar@{-}[d]_{w_{19}}\ar@{-}[r]^{w_{2}}\ar@{-}[dr] & {P_2} \ar@{-}[d]\ar@{-}[dl]_{w_{38} \;\; \; }\\
                        {P_1^{\sigma}} \ar@{-}[r] & {P_2^\sigma} } &
                        
       \xymatrix{{P_3} \ar@{=}[d]^{w_{2},w_{38}} \ar@{-}[r]^{w_{19}} & {P_3}\ar@{=}[d] \\ {P_3^\sigma}   \ar@{-}[r]^{w_2}  & {P_3^\sigma}  }  \\                 
             
        \end{tabular}  
      
        \end{table}
        

     \vspace{1.5cm}

\begin{table}[!htbp]
	 \caption{\boldmath $X_0(42)$}
	 \vspace{-1cm}
	\begin{flalign*} 
\text{Genus :} & 5 \\
\text{Model: } &   x_1x_3 - x_2^2 + x_3x_4, \\
& x_1x_5 - x_2x_5 - x_3^2 + x_4x_5 - x_5^2, \\
& x_1x_4 - x_2x_3 + x_2x_4 - x_3^2 + x_3x_4 + x_3x_5 - x_4^2 -   2x_4x_5 \\
  C \isom  &  \Z/2\Z  \times \Z/2\Z \times \Z/12\Z \times \Z/48\Z \; \text{and\;} J_0(42)(\Q)/C  \isom 0 \; \text{or} \; \Z/2\Z \\
                             \end{flalign*}        
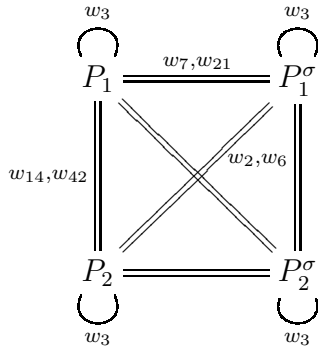
                                     \begin{tabular}{cccccc}
                                     \hline
                                     Name & $\T^2 $& Coordinates & $j$-Invariant & CM by & $\Q$-curve\\
                                     \hline &&&&&\\
                                     $P_1$ & -3& $(\frac{-\T+1}{2}, \frac{\T+1}{2} , \frac{\T-1}{2}, \frac{\T+1}{2} ,1)$ & 54000 & -12 & YES\\[1ex]
                                       $P_2$ & -3& $(2, \frac{\T+1}{2}, \frac{\T-1}{2} , -1 ,1)$ & 0 & -3 & YES\\  [1ex]
                                     \hline
                                     \end{tabular} 
                                     
 \begin{tabular}{c}
                                     
                                           \xymatrix{
                                             {P_1}\ar@{-}@(ul,ur)^{w_3}\ar@{=}[rr]^{w_7,w_{21}} \ar@{=}[dd]_{w_{14},w_{42}}\ar@{=}[ddrr]^{\;\;\;\; w_2,w_{6}}& & {P_1^\sigma} \ar@{=}[dd]\ar@{-}@(ul,ur)^{w_3}\ar@{=}[ddll]\\ & &\\
                                             {P_2}\ar@{-}@(dr,dl)^{w_3}\ar@{=}[rr]& & {P_2^\sigma}\ar@{-}@(dr,dl)^{w_3} } \\                 
                                      
                              \end{tabular}      
                              
                  \end{table}


                       \begin{table}[!htbp]
                              	\caption{\boldmath $X_0(44)$}
                              	\begin{flalign*}
                             &	\text{Genus: \;}   4& \\
                              &	\text{Model: \;}  x_1^2x_4 - x_2^3 + x_3^2x_4 - 2x_4^3,
                              	                        \; \; \;x_1x_3 - x_2^2 + 2x_2x_4 - 3x_4^2 &\\
              &   J_0(44)(\Q) = C \isom \Z/5\Z \times  \Z/5\Z \times \Z/15\Z &
                             	\end{flalign*}
                              	\begin{tabular}{cccccc}
                              		\hline
                              		Name & $\T^2$ & Coordinates & $j$-invariant & CM by & $\Q$-curve\\
                              		\hline \hline &&&&& \\
                              		$P_1$ & -7& $(\frac{-\T+1}{2}, \frac{\T+1}{2} , 1, 1)$ & -3375 &  -7 & YES\\ [1ex]
                              		$P_2$ & -7& $(\frac{\T-1}{2} ,  \frac{\T+1}{2}, -1, 1)$ & 16581375 & -28  & NO\\[1ex]  
                              		$P_3$ & -7 & $(1, \frac{-\T+1}{2} ,  \frac{\T+1}{2},  1)$ & 16581375  & -28 & NO\\ [1ex]
                              		$P_4$ & -7 & $(-1, \frac{\T+1}{2}, \frac{\T-1}{2}, 1)$ &-3375  & -7  & YES\\[1ex]
                              		$P_5$ & -7 &  $(\T-2, -2, -\T-2, 1)$  & -3375  &  -7 & NO \\[1ex]
                              	    $P_6$ & -7 & $(-\T+2, -2, \T+2, 1)$ & -3375  & -7  & NO \\[2ex]
                              	\hline
                              	\end{tabular}
                              	\vspace{0.1cm}
                              	\begin{tabular}{ccc}
                              		
                              		\xymatrix{
                              			{P_1} \ar@{-}[d]_{w_{44}}\ar@{-}[r]^{w_{11}}\ar@{-}[dr]& {P_4} \ar@{-}[d]\ar@{-}[dl]_{w_{4}\;\;\;}\\
                              			{P_1^{\sigma}} \ar@{-}[r] & {P_4^\sigma} } &
                              		
                              		\xymatrix{
                              		{P_2} \ar@{-}[d] _{w_{44}} \ar@{-}[r]^{w_{11}} \ar@{-}[dr] & {P_3^\sigma}\ar@{-}[d]\ar@{-}[dl]_{w_{4}\;\;\;}
                              			\\ {P_5}   \ar@{-}[r] & {P_6^\sigma}  }  &
                              			
                              			\xymatrix{
                              		        		{P_2^\sigma} \ar@{-}[d]_{w_{44}}\ar@{-}[r]^{w_{11}} \ar@{-}[dr] & {P_3}\ar@{-}[d] \ar@{-}[dl]_{w_{4}\;\;\;}
                              		        			\\ {P_5^\sigma}   \ar@{-}[r] & {P_6}  } 
                              	    \\                 
                              		
                              	\end{tabular} 
                              \end{table}
                              

\begin{table}[!htbp]
	\caption{\boldmath $X_0(45)$}
		 \vspace{-1cm}
\begin{flalign*}
& \text{Genus: } 3 \\
& \text{Model: } x^3z - x^2y^2 + xyz^2 - y^3z - 5z^4 \\
& J_0(45)(\Q) = C \isom \Z/2\Z \times  \Z/4\Z \times \Z/8\Z \\
\end{flalign*}
	 \vspace{-1cm}
	 
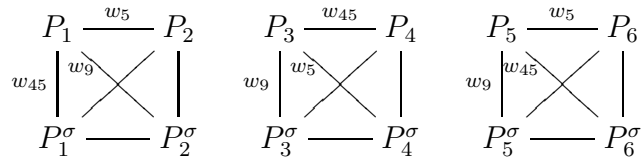
	\begin{tabular}{cccccc}
		\hline
		Name & $\T^2$ & Coordinates & $j$-invariant & CM by & $\Q$-curve\\
		\hline \hline &&&&&\\
		$P_1$ & -11& $(\frac{\T-1}{2}, 1 , 1)$ &  -32768 & -11 & YES\\[1ex]
		$P_2$ & -11& $(-1 ,  \frac{-\T+1}{2}, 1)$ &-32768 & -11 & YES\\[1ex]  
		$P_3$ & 13 & $(2,\frac{-\T-5}{2},1)$ & $\frac{1250637664527933\T - 4509238226399579}{64}$ & NO & YES\\ [1ex]
		$P_4$ & 13 & $( \frac{-\T+5}{2},-2,1)$ & $\frac{461373\T - 1664219}{4}$ & NO & YES\\[1ex]
		$P_5$ & -39 &  $(\frac{\T+13}{8},\frac{-\T-5}{4} ,1)$  & $\frac{2734106225\T + 43419758443 }{1024}$ & NO & YES \\[1ex]
		$P_6$ & -39 & $(\frac{\T+5}{4},\frac{-\T-13}{8},1)$ &  $\frac{-60355066783497695\T - 
		    1556546639145161477}{70368744177664}  $& NO & YES \\[2ex]
	\hline
	\end{tabular}
		\vspace{0.1cm}
	 \begin{tabular}{ccc}
			\xymatrix{
			{P_1} \ar@{-}[d]_{w_{45}}\ar@{-}[r]^{w_{5}}\ar@{-}[dr] & {P_2} \ar@{-}[d]\ar@{-}[dl]_{w_{9}\;\;\;}\\
			{P_1^{\sigma}} \ar@{-}[r] & {P_2^\sigma} } &
		
	\xymatrix{
		{P_3} \ar@{-}[d]_{w_9}\ar@{-}[r]^{w_{45}} \ar@{-}[dr] & {P_4}\ar@{-}[d] \ar@{-}[dl]_{w_{5}\;\;\;}
			\\ {P_3^\sigma}   \ar@{-}[r] & {P_4^\sigma}  }  &
			
		\xymatrix{{P_5} \ar@{-}[d]_{w_{9}}\ar@{-}[r]^{w_{5}}\ar@{-}[dr] & {P_6} \ar@{-}[d]\ar@{-}[dl]_{w_{45}\;\;\;}\\
						{P_5^\sigma} \ar@{-}[r] & {P_6^\sigma}      }  \\

	\end{tabular}  
	\end{table}

	\vspace{-1cm}
           \begin{table}[!htbp]
                               	\caption{\boldmath $X_0(51)$}
                               		\vspace{-1cm}
                               	 \begin{flalign*}
&  \text{Genus: } 5 \\
& \text{Model: } x_1x_3 - x_2^2 + x_2x_4 - x_3^2 - x_3x_5 - x_4^2, \\
                               & 	x_1x_4 - x_2x_3 - x_3^2 - x_4x_5, \\
                               	& x_1x_5 - x_2x_4 - 2x_3^2 + x_3x_5 + x_4^2 - 2x_4x_5 \\
                 & J_0(51)(\Q) = C \isom \Z/8\Z \times  \Z/48\Z \\
                                \end{flalign*} 
                              
                               	\vspace{-0.5cm}
                               	
                               	\begin{tabular}{cccccc}
                               		\hline
                               		Name & $\T^2$ & Coordinates & $j$-invariant & CM by & $\Q$-curve\\
                               		\hline \hline &&&&& \\
                               		$P_1$ & -2& $(0, \T/2, \T/2, 1,1 )$ & 8000 &  -8 & YES\\[1ex]
                               		$P_2$ & -2& $(\frac{-\T+1}{9} ,  \frac{2\T-1}{9}, \frac{2\T-1}{9}, \frac{\T+4}{9}, 1)$ & 8000 & -8  & NO\\[1ex]  
                               		$P_3$ & 17 & $( \frac{\T+5}{2} ,  \frac{-\T-3}{4},\frac{\T+3}{4}, 0,  1)$ & $-671956992\T- 2770550784 $  & -51 & NO\\ [1ex]
                               	
                               	\hline
                               	\end{tabular}
                               	\vspace{0.2cm}
                               	\begin{tabular}{cc}
                               		
                               		\xymatrix{
                               			{P_1} \ar@{-}[d]_{w_{51}}\ar@{-}[r]^{w_{17}}\ar@{-}[dr] & {P_2} \ar@{-}[d]\ar@{-}[dl]_{w_{3}\;\;\;}\\
                               			{P_1^{\sigma}} \ar@{-}[r] & {P_2^\sigma} } &
                               		
                               		\xymatrix{
                               		{P_3}\ar@{-}@(ul,ur)^{w_{51}}  \ar@{=}[r]^{w_3, w_{17}} & {P_3^\sigma}\ar@{-}@(ul,ur)^{w_{51}}  \\ } 
                               			
                               	    \\                 
                               		
                               	\end{tabular} 
                               \end{table}

                          \begin{table}[!htbp]
                          	\caption{\boldmath $X_0(52)$}
                          		\vspace{-1.0cm}
                          	\begin{flalign*}
                          &	\text{Genus: } 5 \\
                & \text{Model: } x_1x_3 - x_2^2 - x_3^2 - x_4^2, \\
              & x_1x_4 - x_2x_3 + x_2x_5 - x_4x_5, \\
                          	& x_1x_5 - x_2x_4 - 2x_3^2 + x_3x_5 - x_5^2 \\
                          & J_0(52)(\Q) = C \isom \Z/21\Z \times  \Z/42\Z \\
                          \end{flalign*}
                       
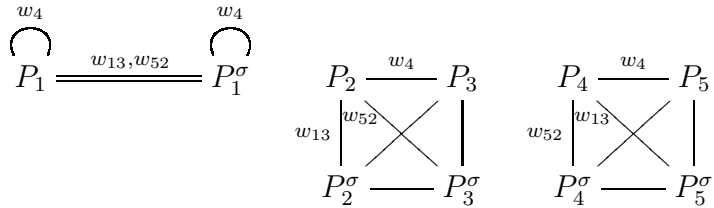
                          	\begin{tabular}{cccccc}
                          		\hline
                          		Name & $\T^2$ & Coordinates & $j$-invariant & CM by & $\Q$-curve\\
                          		\hline \hline &&&&& \\
                          		$P_1$ & -1& $(\T+1,1,0,\T,1)$ & 287496 &  -16 & YES\\
                          		$P_2$ & -1& $(\frac{-\T+1}{2} ,  0, \frac{-\T+1}{2}, 0, 1)$ & 287496& -16  & YES\\[1ex]  
                          		$P_3$ & -1 & $(\T+1, -1, 0 ,-\T,   1)$ & 1728  & -4 & YES\\ 
                          		$P_4$ & -3 & $(0, -1, \frac{\T+1}{2}, \frac{\T-1}{2}, 1)$ & 54000  & -12 & YES\\[1ex]
                          		$P_5$ & -3 &  $(0, 1, \frac{\T+1}{2}, \frac{-\T+1}{2}, 1)$  & 54000  &  -12 & YES \\[1ex]
                          	    
                          	\hline
                          	\end{tabular}
                          
                          	\begin{tabular}{ccc}
                          		\xymatrix{
                          		                       		{P_1}\ar@{-}@(ul,ur)^{w_{4}}  \ar@{=}[rr]^{w_{13}, w_{52}} & & {P_1^\sigma}\ar@{-}@(ul,ur)^{w_4}}  &
                          		
                          		\xymatrix{
                          		{P_2} \ar@{-}[d]_{w_{13}} \ar@{-}[r]^{w_{4}} \ar@{-}[dr] & {P_3}\ar@{-}[d]\ar@{-}[dl]_{w_{52}\;\;\;}
                          			\\ {P_2^\sigma}   \ar@{-}[r] & {P_3^\sigma}  }  &
                          			
                          			\xymatrix{
                          		        		{P_4} \ar@{-}[d]_{w_{52}} \ar@{-}[r]^{w_{4}} \ar@{-}[dr] & {P_5}\ar@{-}[d]\ar@{-}[dl]_{w_{13}\;\;\;}
                          		        			\\ {P_4^\sigma}   \ar@{-}[r] & {P_5^\sigma}  } 
                          	    \\                 
                          		
                          	\end{tabular} 
                          \end{table}
                          

        \begin{table}[!htbp]
                        \caption{\boldmath $X_0(54)$}
                          \begin{flalign*} 
                        &  \text{Genus: } 4 \\
                     &  \text{Model: } x_1^2x_3 - x_1x_3^2 - x_2^3 + x_2^2x_4 - 3x_2x_4^2 + x_3^3 + 3x_4^3,\\
             & x_1x_4 - x_2x_3 + x_3x_4 \\
                      & J_0(54)(\Q) = C \isom \Z/3\Z \times \Z/3\Z \times \Z/9\Z \\
                      \end{flalign*}

                        \begin{tabular}{cccccc}
                        \hline
                        Name & $\T^2 $& Coordinates & $j$-Invariant & CM by & $\Q$-curve\\
                        \hline
                        $P_1$ & -2& $(-2, \T+1 , \T ,1)$ & 8000 & -8 & YES\\
                            
                        \hline
                        \end{tabular}
          
            \vspace{0.2cm}
        
        \begin{tabular}{c}
                        
                              	\xymatrix{
                                                		                       		{P_1}\ar@{-}@(ul,ur)^{w_{2}}  \ar@{=}[rr]^{w_{27}, w_{54}} & & {P_1^\sigma}\ar@{-}@(ul,ur)^{w_2}} \\                 
                         
                 \end{tabular}           
     \end{table}
     
       
         \begin{table}\label{55}
                                     	\caption{\boldmath $X_0(55)$}
                                     	       \begin{flalign*}
                                        & \text{Genus: } 5 \\
                                    &	\text{Model: } x_1x_3 - x_2^2 + x_2x_4 - x_2x_5 - x_3^2 + 3x_3x_4 + x_3x_5 -  2x_4^2 - 4x_5^2, \\
                                     	& 	x_1x_4 - x_2x_3 + 2x_2x_4 - 2x_2x_5 - 2x_3^2 + 4x_3x_4 + 
                                     	    5x_3x_5 - 2x_4^2 - 4x_4x_5 - 3x_5^2, \\
         &  x_1x_5 - 2x_2x_5 - x_3^2 + 2x_3x_4 + x_3x_5 - x_4^2 \\
                          & C \isom \Z/10\Z \times  \Z/20\Z \text{ and }  J_0(55)(\Q)/C \isom  0, \Z/ 2\Z  \text{ or }  (\Z/ 2\Z)^2 \\
                                \end{flalign*} 

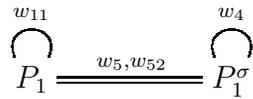
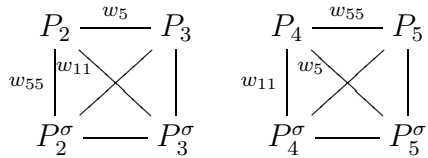
                                     	\begin{tabular}{cccccc}
                                     		\hline
                                     		Name & $\T^2$ & Coordinates & $j$-invariant & CM by & $\Q$-curve\\
                                     		\hline \hline &&&&& \\
                                     		$P_1$ & -11& $(\frac{-\T-1}{2} ,  2,2, \frac{-\T+5}{2}, 1)$  & -32768 &  -11 & YES\\ [1ex]
                                     		$P_2$ & -19& $(\frac{\T-3}{2} ,  \frac{-\T+3}{2}, 2,  \frac{-T+1}{2}, 1)$ & -884736& -19  & YES\\[1ex]  
                                     		$P_3$ & -19 & $(-2, \frac{\T+3}{2}, \frac{-\T+5}{2}, 1,1)$ & -884736 & -19 & YES\\ [1ex]
                                     		$P_4$ & -159 & $(\frac{5\T - 9}{32} , 1/8 ,\frac{-5\T + 19}{32} , \frac{-5\T+27}{32} , 1)$ & $\frac {-306924645775\T - 1607809480031}{4096}$ & NO & YES\\[1ex]
                                     		$P_5$ & -159 & $(\frac{\T + 7}{2},-4,\frac{-\T-5}{2}, \frac{-\T+1}{2},1)$
                                     	  & $\frac{2595124295410999055\T - 214625676177684264671}{72057594037927936}$  &  NO & YES \\[1ex]
                                     	    
                                     	\hline
                                     	\end{tabular}
                                     	\vspace{0.2cm}
                                     	\begin{tabular}{ccc}
                                     		\xymatrix{
                                     		                       		{P_1}\ar@{-}@(ul,ur)^{w_{11}}  \ar@{=}[rr]^{w_{5}, w_{52}} & & {P_1^\sigma}\ar@{-}@(ul,ur)^{w_4}}  &
                                     		
                                     		\xymatrix{
                                     		{P_2} \ar@{-}[d]_{w_{55}} \ar@{-}[r]^{w_{5}} \ar@{-}[dr] & {P_3}\ar@{-}[d]\ar@{-}[dl]_{w_{11}\;\;\;}
                                     			\\ {P_2^\sigma}   \ar@{-}[r] & {P_3^\sigma}  }  &
                                     			
                                     			\xymatrix{
                                     		        		{P_4} \ar@{-}[d]_{w_{11}} \ar@{-}[r]^{w_{55}} \ar@{-}[dr]& {P_5}\ar@{-}[d]\ar@{-}[dl]_{w_{5}\;\;\;} 
                                     		        			\\ {P_4^\sigma}   \ar@{-}[r] & {P_5^\sigma}  } 
                                     	    \\                 
                                     		
                                     	\end{tabular} 
                                     \end{table}
                                  

                                      \begin{table}[!htbp]
                                      	\caption{\boldmath $X_0(56)$}
                                      	   	\begin{flalign*} 
                                      	   	&\text{Genus: }5 \\
                                      & \text{Model: } x_1x_3 - x_2^2 - x_3^2 + x_3x_5 - 3x_4^2,\\
                                      	& x_1x_4 - x_2x_3 + x_2x_5 - x_3x_4, \\
                                      	& x_1x_5 - x_2x_4 - x_3^2 + 2x_3x_5 - x_4^2 - x_5^2 \\
                              & J_0(56)(\Q ) = C \isom \Z/2\Z \times \Z/ 6\Z  \times \Z/ 6\Z \times  \Z/24\Z \\
                              \end{flalign*}


                                      	\begin{tabular}{ c ccccc}
                                      		\hline
                                      		\thead{Name} & \thead{$\T^2$} & \thead{Coordinates} & \thead{$j$-invariant} & \thead{CM by} & \thead{$\Q$-curve}\\
                                      		\hline \hline &&&&&\\
                                      		$P_1$ & -7& $(-1, \frac{3\T+7}{8} ,  \frac{-\T+3}{4}, \frac{\T-3}{8}, 1)$  & -3375 &  -7 & NO\\[1ex]
                                      		$P_2$ & -7& $(-1, \frac{-\T+1}{2} ,  \frac{\T+1}{2}, -1,  1)$ & 16581375 & -28 & NO\\[1ex]  
                                      		$P_3$ & -7 & $(-1, \frac{-3\T-7}{8}, \frac{-\T+3}{4}, \frac{-\T+3}{8},1)$ & 16581375 & -28 & YES\\ [1ex]
                                      		$P_4$ & -7 & $(-1, \frac{\T - 1}{2} , \frac{\T + 1}{2} , 1, 1)$ & -3375  & -7 & YES\\[1ex]

                                      	\hline
                                      	\end{tabular}
                                      	\vspace{0.2cm}
                                      	\begin{tabular}{cccc}
                                      		\xymatrix{
                                      		                       		{P_1}\ar@{-}@(ul,ur)^{w_{7}}  \ar@{=}[rr]^{w_{8}, w_{56}} & & {P_2^\sigma}\ar@{-}@(ul,ur)^{w_7}}  &
                                      		
                                   	\xymatrix{
                                                                  		                       		{P_2}\ar@{-}@(ul,ur)^{w_{7}}  \ar@{=}[rr]^{w_{8}, w_{56}} & & {P_1^\sigma}\ar@{-}@(ul,ur)^{w_7}}  &

                                          	\xymatrix{
                                   		{P_3}\ar@{-}@(ul,ur)^{w_{7}}  \ar@{=}[rr]^{w_{8}, w_{56}} & & {P_3^\sigma}\ar@{-}@(ul,ur)^{w_7}}  &  
                                      			
                                      				\xymatrix{
                                      			                               		                       		{P_4}\ar@{-}@(ul,ur)^{w_{7}}  \ar@{=}[rr]^{w_{8}, w_{56}} & & {P_4^\sigma}\ar@{-}@(ul,ur)^{w_7}}  
                                      	    \\                 
                                      		
                                      	\end{tabular} 
                                      \end{table}
                                   
            \begin{table}[!htbp]
                                         	\caption{\boldmath $X_0(63)$}
                                         	\vspace{-1cm}
                                         	\begin{flalign*}
                                        &  \text{Genus: } 5 \\
                                       &  \text{Model: } x_1x_3 - x_2^2 + x_2x_5 - x_3x_4 - x_5^2,\\
                                         	& x_1x_4 - x_2x_3 - x_3x_5,\\
                                        & x_1x_5 - x_2x_4 - x_3^2 \\
                                  & C \isom \Z/2\Z \times  \Z/4\Z \times  \Z/48\Z \text{ and }  J_0(63)(\Q)/C  \isom 0 \text{ or }  \Z/ 2\Z\\
                                  	\end{flalign*} 
                                         	
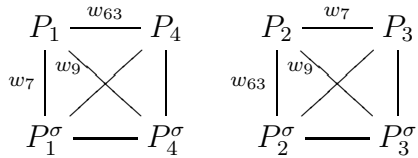
                                         	\begin{tabular}{cccccc}
                                         		\hline
                                         		Name & $\T^2$ & Coordinates & $j$-invariant & CM by & $\Q$-curve\\
                                         		\hline \hline &&&&&\\
                                         		$P_1$ & -3 & $\left(\frac{-\T+1}{2}, \frac{-\T-1}{2}, \frac{\T-1}{2}, \frac{\T-1}{2},1\right )$ & 0 &  -3 & YES\\[1ex]
                                         		$P_2$ & -3&  $\left (\frac{\T+1}{2}, \frac{-\T-1}{2}, 1,  \frac{-\T-1}{2} ,1 \right )$ & -12288000 & -27  & YES\\[1ex]  
                                         		$P_3$ & -3 &  $\left (-1, \frac{\T-1}{2}, \frac{\T-1}{2}, 1 ,1 \right )$ &  -12288000 & -27 & YES\\ [1ex]
                                         		$P_4$ & -3 & $(0, \frac{-\T+1}{2}, 0,0, 1)$ & -12288000 & -27 & YES\\[1ex]

                                         	\hline
                                         	\end{tabular}
                                         	\vspace{0.2cm}
                                         	\begin{tabular}{cc}
                                         
                                         \xymatrix{
                                                                         		{P_1} \ar@{-}[d]_{w_{7}} \ar@{-}[r]^{w_{63}} \ar@{-}[dr] & {P_4}\ar@{-}[d]\ar@{-}[dl]_{w_{9}\;\;\;}
                                                                         			\\ {P_1^\sigma}   \ar@{-}[r] & {P_4^\sigma}  } 
                                                                         			
                                                                         				  &
                                         		
                                         		\xymatrix{
                                         		{P_2} \ar@{-}[d]_{w_{63}} \ar@{-}[r]^{w_{7}} \ar@{-}[dr] & {P_3}\ar@{-}[d]\ar@{-}[dl]_{w_{9}\;\;\;}
                                         			\\ {P_2^\sigma}   \ar@{-}[r] & {P_3^\sigma}  }  \\

                                         	\end{tabular} 
                                         \end{table}


\begin{table}[!htbp]
	\caption{\boldmath$X_0(64)$}
	\vspace{-1cm}
	\begin{flalign*}
& \text{Genus: } 3 \\
& \text{Model: } x^3z + 4xz^3 - y^4 \\
& J_0(64)(\Q) =C\isom \Z/2\Z \times \Z/4\Z \times \Z/4\Z \\
\end{flalign*}


	\begin{tabular}{cccccc}
			\hline
		Name & $\T^2$ & Coordinates & $j$-invariant & CM by & $\Q$-curve\\
		\hline \hline
		$P_1$ & -7& $(-\T-1,-2, 1)$ &  -3375 & -7 & NO\\
		$P_2$ & -7& $(-\T-1 , 2,1 )$ & 16581375& -28 & YES\\
		$P_3$ & -7 & $(\frac{-\T-1}{2},\frac{\T+1}{2}, 1)$ & -3375 & -7& YES\\ [1ex]  
		$P_4$ & -7 & $( \frac{\T-1}{2},\frac{\T-1}{2},1)$ & 16581375 & -28 & NO\\[1ex]
		\hline
	\end{tabular}
	
		\vspace{0.5cm}
		
	 \begin{tabular}{ccc}
		\xymatrix{
			{P_1} \ar@{-}[r]^{w_{64}}& {P_4^\sigma}  } &
		
    	\xymatrix{
				{P_2} \ar@{-}[r]^{w_{64}}& {P_2^\sigma}  } &
			
	   \xymatrix{
							{P_3} \ar@{-}[r]^{w_{64}}& {P_3^\sigma}  }  \\                 
		\end{tabular}  
	\end{table}


                                  \begin{table}[!htbp]
                                  	\caption{\boldmath $X_0(72)$}
                                  	    \vspace{-1cm}
                                  	    \begin{flalign*}
                                  	    & \text{Genus: } 5\\
                                  &	\text{Model: } x_1x_3 - x_2^2 - x_4^2 - 4x_5^2,\\
                         & x_1x_4 - x_2x_3 + x_2x_5 + x_4x_5 - 3x_5^2, \\
                                  	& x_1x_5 - x_2x_4 - x_3x_5\\
                          & C \isom \Z/2\Z \times  \Z/4\Z \times  \Z/12\Z \times  \Z/12\Z  \text{ and } 
                                   J_0(72)(\Q )/C \isom 0 \; \text{ or } \; \Z/2\Z  \\
                                   \end{flalign*}
                                 \vspace{-1cm}
                               
                             
                             There are no non-cuspidal quadratic points on $X_0(72)$.
                             \end{table}

                                            \begin{table}[!htbp]
                                            	\caption{\boldmath $X_0(75)$}
                                            	       Genus: 5   
                                            	       
                                            	Model: $x_1x_3 - x_2^2 + x_2x_5 - x_3x_4 - x_5^2$,
                                            	
                                            	\hspace{1.2cm} $x_1x_4 - x_2x_3 - x_3x_5$,
                                            	
                                            	\hspace{1.2cm} $x_1x_5 - x_2x_4 - x_3^2$
                                            	
                                         $C \isom \Z/2\Z \times  \Z/4\Z \times \Z/40\Z$ and $J_0(75)(\Q )/C\isom 0$ or  $\Z/2\Z$ or $\Z/4\Z$\\


                                            	\begin{tabular}{cccccc}
                                            		\hline
                                            		Name & $\T^2$ & Coordinates & $j$-invariant & CM by & $\Q$-curve\\
                                            		\hline \hline &&&&&\\
                                            		$P_1$ & 5 & $\left (\frac{-3\T + 1}{2} , \frac{\T-3}{2}, \frac{-3\T-1}{2}, \frac{\T-3}{2}, 1\right ) $ & $146329141248\T - 327201914880$ &  -75 & YES\\[1ex]
                                            		$P_2$ & -11&  $\left (-2 , \frac{t - 1}{2},  \frac{t + 1}{2} , \frac{t - 1}{2} , 1\right )$ & -12288000 & -11  & YES\\[1ex]  
                                            		$P_3$ & -11 &  $\left (\frac{-\T- 1}{2} , \frac{\T -  1}{2} , 2 ,\frac{\T - 1}{2} , 1\right ) $&  -12288000 & -11 & YES\\ [1ex]

                                            	\hline
                                            	\end{tabular}
                                            	\vspace{0.2cm}
                                            	\begin{tabular}{cc}
                                            
                                    	\xymatrix{
                                                                		{P_1}\ar@{-}@(ul,ur)^{w_{75}}  \ar@{=}[rr]^{w_{3}, w_{25}} & & {P_1^\sigma}\ar@{-}@(ul,ur)^{w_{75}}}

                                                                            				  &
                                            		
                                            		\xymatrix{
                                            		{P_2} \ar@{-}[d]_{w_{75}} \ar@{-}[r]^{w_{3}} \ar@{-}[dr]& {P_3}\ar@{-}[d]\ar@{-}[dl]_{w_{25}\;\;\;} 
                                            			\\ {P_2^\sigma}   \ar@{-}[r] & {P_3^\sigma}  }  \\

                                            	\end{tabular} 
                                            \end{table}


 \begin{table}[!htbp]
                             \caption{\boldmath $X_0(81)$}
                                     \begin{flalign*}
                                   &  \text{Genus: } 4 \\
                           & \text{Model: } x_1^2x_4 - x_1x_4^2 - x_2^3 - 3x_3^3 + x_4^3,\\
                             & x_1x_3 - x_2^2 - 2x_3x_4 \\
                            & J_0(81)(\Q) = C \isom \mathbb Z/ 3 \Z \times \Z/ 9\Z\\
                            \end{flalign*}


                             \begin{tabular}{cccccc}
                             \hline
                             Name & $\T^2 $& Coordinates & $j$-Invariant & CM by & $\Q$-curve\\
                             \hline
                             $P_1$ & -2& $(\frac{-\T}{2}, \frac{\T-1}{2} , \frac{\T}{2}, ,1)$ & 8000 & -8 & YES\\
                               $P_2$ & -11& $(\frac{-\T-1}{2}, \frac{-\T+1}{2} , 1,1)$ & -32768 & -11 & YES\\
                             \hline
                             \end{tabular}
               
                 \vspace{0.2cm}
             
             \begin{tabular}{cc}
                             
                                   \xymatrix{
                                     {P_1}\ar@{-}[r]^{w_{81}}& {P_1^\sigma}  } 
                                     
                                     &
                                     
                                     \xymatrix{
                                                                          {P_2}\ar@{-}[r]^{w_{81}}& {P_2^\sigma}  }
                                      \\                 
                              
                      \end{tabular}           
          \end{table}

\end{document}